\documentclass[12pt]{amsart}
 
\pdfoutput=1
\usepackage[margin=1in]{geometry}
\usepackage{amsmath,amsthm,amssymb,amsrefs,bbm,color,esint,esvect,float,graphicx,mathrsfs}
\usepackage{soul}

\usepackage{euscript,latexsym}
\usepackage{accents}

\usepackage{epstopdf}
\AppendGraphicsExtensions{.tif}
\newcommand{\N}{\mathbb{N}}

\newcommand{\R}{\mathbb{R}}
\newcommand{\C}{\mathbb{C}}
\newcommand{\D}{\mathcal{D}}
\newcommand{\W}{\mathcal{W}}
\newcommand{\HS}{\mathcal{H}}
\newcommand{\Ptens}{\widehat{\otimes}}

 \newcommand{\BB}{\mathcal{B}}   \newcommand{\DD}{\mathcal{D}}

 \newcommand{\vertiii}[1]{{\left\vert\kern-0.25ex\left\vert\kern-0.25ex\left\vert #1 
    \right\vert\kern-0.25ex\right\vert\kern-0.25ex\right\vert}}

 \newtheorem{thm}{Theorem}[section]
 \newtheorem{lemma}[thm]{Lemma}
 \newtheorem{cor}[thm]{Corollary}
 \newtheorem{prop}[thm]{Proposition}
 
 \newtheorem{defin}[thm]{Definition}

 \numberwithin{equation}{section}
 
\newenvironment{theorem}[2][Theorem]{\begin{trivlist}
\item[\hskip \labelsep {\bfseries #1}\hskip \labelsep {\bfseries #2.}]}{\end{trivlist}}

\theoremstyle{definition}

\begin{document}

\title[On the compactness of bi-parameter singular integrals]{On the compactness of bi-parameter singular integrals}

\author{Cody B. Stockdale}
\address{Cody B. Stockdale, School of Mathematical Sciences and Statistics, Clemson University, Clemson, SC 29634, USA}
\email{cbstock@clemson.edu}

\author{Cody Waters}
\thanks{C. Waters is supported by the National Science Foundation, Grant No. DGE 2139839}
\address{Cody Waters, Department of Mathematics, Washington University in St. Louis, St. Louis, MO 63130, USA}
\email{codyw@wustl.edu}

\begin{abstract}
    We establish a new $T1$ theorem for the compactness of bi-parameter Calder\'on-Zygmund singular integral operators. Namely, we show that if a bi-parameter CZO $T$ satisfies the product weak compactness property, the mixed weak compactness/CMO property, and $T1, T^t1,$ $T_t1, T_t^t1 \in \text{CMO}(\mathbb{R}^{n_1}\times\mathbb{R}^{n_2})$, then $T$ is compact on $L^2(\mathbb{R}^{n_1}\times\mathbb{R}^{n_2})$. We also obtain endpoint compactness results for these operators and use them to deduce the necessity of most of our hypotheses. In particular, our conditions characterize the simultaneous $L^2(\mathbb{R}^{n_1}\times\mathbb{R}^{n_2})$-compactness of a bi-parameter CZO and its partial transpose. Our assumptions improve upon previously known sufficient conditions, and our proof, which is shorter and simpler than earlier arguments, utilizes a new abstract compactness criterion for partially localized operators on tensor products of Hilbert spaces. 
\end{abstract}

\keywords{Singular integral operator, bi-parameter operator, compact operator, $T1$ theorem}

\subjclass[2020]{42B20, 47B07}

\maketitle


\section{Introduction}\label{Section:Introduction}

The Calder\'on-Zygmund theory of singular integral operators (SIOs) is a central topic in harmonic analysis. A main cornerstone of this field is David and Journ\'e's celebrated $T1$ theorem of \cite{DJ1984}, which characterizes the $L^2(\mathbb{R}^n)$-bounded SIOs $T$ in terms of the weak boundedness property and the membership $T1,\, T^{t}1 \in \text{BMO}(\mathbb{R}^n)$ -- such SIOs are called Calder\'on-Zygmund operators (CZOs). The classical theory further yields boundedness of CZOs on weighted spaces $L^p(w)$ for $p \in (1,\infty)$ and weights $w\in A_p$, as well as endpoint estimates from $H^1(\mathbb{R}^n)$ to $L^1(\mathbb{R}^n)$ and, by duality, from $L^\infty(\mathbb{R}^n)$ to $\text{BMO}(\mathbb{R}^n)$. In this work, we investigate bi-parameter, compact analogues of these results.

The study of bi-parameter singular integrals was initiated in the convolution case by Fefferman and Stein in \cite{FS1982} and extended to general non-convolution operators by Journ\'e with the following bi-parameter $T1$-type theorem in \cite{J1985}*{Theorem 4}.
\begin{theorem}{A}
\emph{If a bi-parameter SIO $T$ satisfies the weak boundedness property and $T1, T^t1,$ $T_t1, T_t^t1 \in \text{BMO}(\mathbb{R}^{n_1}\times\mathbb{R}^{n_2})$, then $T$ is bounded on $L^2(\mathbb{R}^{n_1} \times \R^{n_2})$.}
\end{theorem}
\noindent Although the assumptions $T_t, T_t^t \in \text{BMO}(\mathbb{R}^{n_1}\times\mathbb{R}^{n_2})$ are not necessary for general $T$ to be bounded on $L^2(\mathbb{R}^{n_1}\times\mathbb{R}^{n_2})$, the conditions imposed in Theorem A characterize the simultaneous boundedness of $T$ and its partial transpose $T_t$; see \cite{J1985}*{Corollary on p. 76}. This point of view, which is further motivated in \cite{J1985}*{p. 75--76}, leads one to define a bi-parameter CZO to be an $L^2(\mathbb{R} ^{n_1}\times\mathbb {R}^{n_2})$-bounded bi-parameter SIO whose partial transpose is also bounded. 

While the $T1$ theorem of \cite{DJ1984} concerns the boundedness of Calder\'on-Zygmund singular integrals, Villarroya first addressed their $L^2(\mathbb{R}^n)$-compactness in \cite{V2015}. The following updated version of the compact $T1$ theorem was obtained by Mitkovski and the first author in \cite{MS2023}. \looseness=-1
\begin{theorem}{B}
\emph{Let $T$ be a CZO. Then $T$ is compact on $L^2(\mathbb{R}^n)$ if and only if $T$ satisfies the weak compactness property and $T1,T^t1\in\text{CMO}(\mathbb{R}^n)$.} 
\end{theorem}
\noindent A full theory of compact CZOs has by now been developed. In particular, the compactness on $L^2(\mathbb{R}^n)$ extends to $L^p(\mathbb{R}^n)$ for all $p \in (1,\infty)$ \cite{V2015}, to the corresponding endpoints \cites{OV2017, PPV2017}, and to weighted Lebesgue spaces with respect to $A_p$ weights \cites{SVW2022, HL2023, COY2022, FGW2023}. See also \cites{BOT2024, BLOT2025, CST2023} for applications and examples of compact CZOs in the context of pseudodifferential operators. 

These two lines of work were recently connected by Cao, Yabuta, and Yang in \cite{CYY2024} with the following $T1$-type theorem for compactness of bi-parameter CZOs.
\begin{theorem}{C}
\emph{If $T$ is a bi-parameter CZO such that
\begin{enumerate}\addtolength{\itemsep}{0.1cm}
    \item $T$ satisfies the product weak compactness property, 
    \item $T$ admits the compact full kernel representation, 
    \item $T$ admits the compact partial kernel representation, 
    \item $T$ satisfies the diagonal $\text{CMO}$ property, and 
    \item $T1, T^t1, T_t1, T_t^t1 \in \text{CMO}(\mathbb{R}^{n_1}\times \mathbb{R}^{n_2})$,
\end{enumerate}
then $T$ is compact on $L^2(\mathbb{R}^{n_1} \times \R^{n_2})$.}
\end{theorem}

The purpose of this article is to further develop the theory of compact bi-parameter CZOs. Our first main result is the following new version of the compact bi-parameter $T1$ theorem. \looseness=-1
\begin{thm}\label{bi-parameterCZOCompactnessTheorem}
If $T$ is a bi-parameter CZO such that 
\begin{enumerate}\addtolength{\itemsep}{0.1cm}
    \item $T$ satisfies the product weak compactness property,
    \item $T$ satisfies the mixed weak compactness/CMO property, and 
    \item $T1, T^t1, T_t1, T_t^t1 \in \text{CMO}(\mathbb{R}^{n_1}\times \mathbb{R}^{n_2})$,
\end{enumerate}
then $T$ is compact on $L^2(\mathbb{R}^{n_1}\times \mathbb{R}^{n_2})$. Moreover, the product weak compactness property, the mixed weak compactness/CMO property, and the membership $T1, T^t1 \in \text{CMO}(\mathbb{R}^{n_1}\times\mathbb{R}^{n_2})$ are necessary for the $L^2(\mathbb{R}^{n_1}\times\mathbb{R}^{n_2})$-compactness of $T$. 
\end{thm}
\noindent We emphasize that Theorem \ref{bi-parameterCZOCompactnessTheorem} removes the compact full kernel representation assumption from Theorem C and replaces the compact partial kernel representation and the diagonal CMO condition with the mixed weak compactness/CMO property, which refines the former hypotheses. Moreover, we show the necessity of conditions in Theorem \ref{bi-parameterCZOCompactnessTheorem}, and our proof of their sufficiency is shorter and simpler than the argument for Theorem C from \cite{CYY2024}.

Similarly to the situation of Theorem A, the conditions of Theorem \ref{bi-parameterCZOCompactnessTheorem} characterize the simultaneous compactness of $T$ and its partial transpose. 
\begin{cor}\label{characterization}
    Let $T$ be a bi-parameter CZO. Then $T$ and $T_t$ are compact on $L^2(\R^{n_1} \times \R^{n_2})$ if and only if
    \begin{enumerate}\addtolength{\itemsep}{0.1cm}
    \item $T$ satisfies the product weak compactness property,
    \item $T$ satisfies the mixed weak compactness/CMO property, and 
    \item $T1, T^t1, T_t1, T_t^t1 \in \text{CMO}(\mathbb{R}^{n_1}\times \mathbb{R}^{n_2})$.
\end{enumerate}
\end{cor}

We also establish the following bi-parameter analogues of the endpoint properties of \cite{PPV2017}. 
\begin{thm}\label{EndpointCompactness}
    If $T$ is an $L^2(\mathbb{R}^{n_1}\times\mathbb{R}^{n_2})$-compact bi-parameter CZO, then $T$ is compact from $H^1(\R^{n_1} \times \R^{n_2})$ to $L^1(\R^{n_1} \times \R^{n_2})$ and from $L^\infty(\R^{n_1} \times \R^{n_2})$ to $\text{CMO}(\mathbb{R}^{n_1}\times \mathbb{R}^{n_2})$.
\end{thm}
\noindent 
The proof of Theorem \ref{EndpointCompactness} follows by interpolating compactness on $L^2(\mathbb{R}^{n_1}\times\mathbb{R}^{n_2})$ and boundedness from $H^p(\mathbb{R}^{n_1}\times\mathbb{R}^{n_2})$ to $L^p(\mathbb{R}^{n_1} \times \mathbb{R}^{n_2})$ for sufficiently large $p<1$, and duality.\looseness=-1

As observed in \cite{CYY2024}*{p. 6262}, the compactness of compact bi-parameter CZOs automatically extends to $L^p(w)$ for any $p \in (1,\infty)$ and any $w \in A_p(\mathbb{R}^{n_1}\times \mathbb{R}^{n_2})$ by the arguments of \cites{COY2022, HL2023}. We similarly conclude the following as a consequence of Theorem \ref{bi-parameterCZOCompactnessTheorem}. \looseness=-1
\begin{cor}\label{WeightedCompactnessCorollary}
If $p \in (1,\infty)$, $w \in A_p(\R^{n_1}\times\R^{n_2})$, and $T$ is a bi-parameter CZO such that 
\begin{enumerate}\addtolength{\itemsep}{0.1cm}
    \item $T$ satisfies the product weak compactness property,
    \item $T$ satisfies the mixed weak compactness/CMO property, and 
    \item $T1, T^t1, T_t1, T_t^t1 \in \text{CMO}(\mathbb{R}^{n_1}\times \mathbb{R}^{n_2})$,
\end{enumerate}
then $T$ is compact on $L^p(w)$.
\end{cor}

The necessity direction of Theorem \ref{bi-parameterCZOCompactnessTheorem} follows from Cauchy-Schwarz for the product weak compactness property, Theorem \ref{EndpointCompactness} for the condition $T1, T^t1 \in \text{CMO}(\mathbb{R}^{n_1}\times \mathbb{R}^{n_2})$, and a similar argument to the proof of Theorem \ref{EndpointCompactness} for the mixed weak compactness/CMO property. To establish the sufficiency portion, we first subtract four top-level paraproducts from $T$ to produce a top-level cancellative bi-parameter CZO, i.e., a bi-parameter CZO $T'$ satisfying \looseness=-1
$$
    \langle T'(1 \otimes 1), \cdot \otimes \cdot\rangle \,  = \langle T'(1 \otimes \cdot), \cdot \otimes 1\rangle = \langle T'(\cdot \otimes 1), 1 \otimes \cdot\rangle
    = \langle T'(\cdot \otimes \cdot), 1 \otimes 1\rangle = 0.
$$
The compactness of the top-level paraproducts is proved by a reduction to compactness of one-parameter paraproducts, using the hypotheses $T1, T^t1,$ $T_t1, T_t^t1 \in \text{CMO}(\mathbb{R}^{n_1}\times\mathbb{R}^{n_2})$, and known estimates for boundedness of bi-parameter paraproducts. We next further subtract four partial paraproducts to produce a fully cancellative bi-parameter CZO, i.e., a top-level cancellative bi-parameter CZO $T''$ that satisfies 
$$
    \langle T''(1 \otimes \cdot) , \cdot \otimes \cdot\rangle = \langle T''(\cdot \otimes 1) , \cdot \otimes \cdot\rangle = \langle T''(\cdot \otimes \cdot) , 1 \otimes \cdot\rangle = \langle T''(\cdot \otimes \cdot) , \cdot \otimes 1\rangle = 0.
$$ 
We note how this decomposition of a bi-parameter CZO into a sum of a fully cancellative bi-parameter CZO, top-level paraproducts, and partial paraproducts naturally resembles the bi-parameter representation theorem of \cite{DWW2023}*{Theorem on p. 1833}. 

The compactness of the partial paraproducts and the remaining fully cancellative bi-parameter CZO is derived from the following abstract result for partially localized operators. 
\begin{thm}\label{Abstractbi-parameterCompactnessTheorem}
    Let $H_1$ and $H_2$ be Hilbert spaces and $T$ be a partially localized linear operator on $H_1\Ptens H_2$. If $T$ is weakly compact, then $T$ is compact.
\end{thm}
\noindent Above, $H_1\Ptens H_2$ denotes the completion of the algebraic tensor product of $H_1$ and $H_2$, the notion of partial localization quantifies the off-diagonal decay of the matrix coefficients of the operator with respect to a suitable frame for $H_2$ and treating $H_1$ like scalars, and the weak compactness condition asserts that these matrix coefficients, which are operators on $H_1$, are compact and vanish in the direction of the diagonal.  We complete the proof of Theorem \ref{bi-parameterCZOCompactnessTheorem} by showing that the remaining fully cancellative bi-parameter CZO and the partial paraproducts are partially localized with respect to appropriate wavelet frames, and deducing their compactness from Theorem \ref{Abstractbi-parameterCompactnessTheorem} using the product weak compactness property and the mixed weak compactness/CMO property, respectively. 

The paper is organized as follows. In Section \ref{PreliminariesSection}, we define bi-parameter CZOs, the product weak compactness property, product CMO, and the mixed weak compactness/CMO property, and we compare Theorem \ref{bi-parameterCZOCompactnessTheorem} to Theorem C. In Section \ref{EndpointSection}, we prove the endpoint result, Theorem \ref{EndpointCompactness}. In Section \ref{AbstractSection}, we define partially localized operators and prove Theorem \ref{Abstractbi-parameterCompactnessTheorem}. In Section \ref{ParaproductSection}, we discuss the compactness of top-level and partial paraproducts. In Section \ref{FullyCancellativeSection}, we characterize the compactness of fully cancellative bi-parameter CZOs. In Section \ref{MainSection}, we prove our main results, Theorem \ref{bi-parameterCZOCompactnessTheorem} and Corollary \ref{characterization}.


\section{Preliminaries}\label{PreliminariesSection}

We use the notation $A\lesssim B$ if $A\leq CB$ for some constant $C>0$, and write $A\approx B$ if $A\lesssim B$ and $B \lesssim A$. For $x \in \mathbb{R}^n$ and $r>0$, we set $B(x,r)$ to be the Euclidean ball centered at $x$ of radius $r$. We use $|A|$ to denote the Lebesgue measure of a set $A\subseteq \mathbb{R}^n$. We write $f_1 \otimes f_2(x) = f_1(x_1)f_2(x_2)$ for functions $f_i$ on $\mathbb{R}^{n_i}$ and $x=(x_1,x_2) \in \mathbb{R}^{n_1}\times\mathbb{R}^{n_2}$. 

\subsection{Bi-parameter singular integral operators}
In \cite{J1985}, Journ\'e first defined bi-parameter SIOs using vector-valued one-parameter kernels. Later on, the setup of bi-parameter singular integrals was reformulated in terms of mixed-type scalar conditions on bi-parameter kernels by Pott and Villarroya, and by Martikainen, in \cites{PV2011, M2012} -- this viewpoint has also been utilized in \cites{O2015, O2017, CYY2024}; in particular, in Theorem C. 
However, it was shown by Grau de la Herr\'an in \cite{G2016} that these perspectives yield the same class of operators; see also \cite{O2017}*{p. 329}. For convenience in our approach, we adopt Journé's original vector-valued perspective. 

For a Banach space $B$ and $\delta \in (0,1]$, we say that $K\colon \mathbb{R}^{2n}\setminus \{(x,x)\colon x \in \mathbb{R}^n\}\rightarrow B$ 
is a $B$-valued $\delta$-CZ kernel if there exists $C>0$ such that
\begin{align}\label{kernelsize}
\|K(x,y)\|_B\leq \frac{C}{|x-y|^n}
\end{align}
whenever $x \neq y$, and 
\begin{align}\label{kernelsmoothness}
    \|K(x,y)-K(x,y')\|_B, \,\|K(y,x)-K(y',x)\|_B\leq C\frac{|y-y'|^{\delta}}{|x-y|^{n+\delta}}
\end{align}
whenever $|y-y'| \leq \frac{1}{2}|x-y|$. We denote the smallest $C>0$ such that both \eqref{kernelsize} and \eqref{kernelsmoothness} hold by $\|K\|_{B,\delta}$. If $B=\mathbb{C}$, we simply call $K$ a $\delta$-CZ kernel and write $\|K\|_{\delta}$ for $\|K\|_{\mathbb{C},\delta}$. 

Let $\DD(\R^n)$ denote the space of smooth functions with compact support, and let $\DD(\R^n)'$ be the space of distributions. A $\delta$-SIO is a continuous operator $T\colon \DD(\R^n)\to \DD(\R^n)'$ such that 
$$
    \langle Tf,g\rangle =\iint_{\mathbb{R}^{2n}}K(x,y)f(y)g(x)\,dxdy
$$
for some $\delta$-CZ kernel $K$ and all $f, g \in \DD(\mathbb{R}^n)$ with disjoint supports. We say that a $\delta$-SIO is a $\delta$-CZO if it has a bounded extension on $L^2(\mathbb{R}^n)$. Recall that the collection of $\delta$-CZOs, $\delta\text{CZO}(\mathbb{R}^n)$, is a Banach space with norm given for $T \in \delta\text{CZO}(\R^n)$ by 
$$
    \|T\|_{\delta\text{CZO}(\R^n)} := \|T\|_{L^2(\mathbb{R}^n)\rightarrow L^2(\mathbb{R}^n)} + \|K\|_{\delta}.
$$

A continuous operator $T\colon \mathcal{D}(\R^{n_1}) \otimes \mathcal{D}(\R^{n_2}) \rightarrow (\mathcal{D}(\R^{n_1}) \otimes \mathcal{D}(\R^{n_2}))'$ is a bi-parameter $\delta$-SIO if there are $\delta\text{CZO}(\mathbb{R}^{n_j})$-valued $\delta$-CZ kernels $K_i \colon \R^{2n_i} \setminus \{(x,x)\colon x \in \mathbb{R}^{n_i}\} \rightarrow \delta\text{CZO}(\mathbb{R}^{n_j})$ with \looseness=-1
$$
    \langle T(f_1 \otimes f_2), g_1 \otimes g_2\rangle = \iint_{\mathbb{R}^{2n_i}}f_i(y)g_i(x)\langle K_i(x,y)f_j,g_j\rangle\,dxdy
$$
whenever $f_i, g_i \in \mathcal{D}(\mathbb{R}^{n_i})$ have disjoint supports for each $i, j \in\{1,2\}$ with $i \neq j$. The transposes of $T$ are defined through 
\begin{align*}
    \langle T(f_1 \otimes f_2), g_1 \otimes g_2 \rangle &= \langle T^t(g_1 \otimes g_2), f_1 \otimes f_2 \rangle\\
    &=\langle T_t(g_1 \otimes f_2), f_1 \otimes g_2 \rangle\\
    &=\langle T_t^t(f_1 \otimes g_2), g_1 \otimes f_2 \rangle.
\end{align*}
We say that a bi-parameter $\delta$-SIO $T$ is a bi-parameter $\delta$-CZO if 
\begin{align*}
    \|T\|_{\delta \text{CZO}(\R^{n_1} \times \R^{n_2})} := \|T&\|_{L^2(\R^{n_1} \times \R^{n_2}) \rightarrow L^2(\R^{n_1} \times \R^{n_2})} + \|T_t\|_{L^2(\R^{n_1} \times \R^{n_2}) \rightarrow L^2(\R^{n_1} \times \R^{n_2})} \\
    &\quad\quad+ \sum_{\substack{i,j \in \{1,2\}\\ i\neq j}}\|K_i\|_{\delta\text{CZO}(\mathbb{R}^{n_j}), \delta}<\infty. 
\end{align*}

\subsection{Product weak compactness, $\text{CMO}$, and mixed weak compactness/CMO}
To define the product weak compactness property, we follow the approach of \cite{MS2023} in the one-parameter setting and first introduce a continuous wavelet frame indexed by the $ax+b$ group. Recall that the $ax+b$ group is the group $\mathbb{R}^{n+1}_+:=\mathbb{R}^n\times\mathbb{R}_+$ with operation given by \looseness=-1
$$
    (x,t)*(y,s) := (ty+x,ts)
$$
for $x,y \in \mathbb{R}^{n}$ and $t,s >0$. 
The identity element of $\mathbb{R}^{n+1}_+$ is $(0,1)$, and inverses are given by $(x,t)^{-1} = \left(-\frac{x}{t},\frac{1}{t}\right)$.
The left Haar measure, $\lambda$, is given by
$$
    d\lambda(z) = \frac{dxdt}{t^{n+1}}
$$ 
for $z=(x,t) \in \mathbb{R}^{n+1}_+$, and the left-invariant metric, d, is the Riemannian metric given by 
$$
    ds^2=\frac{dx^2+dt^2}{t^2}.
$$
Let $D(z,r) := \{w \in \mathbb{R}^{n+1}_+\colon d(z,w)<r\}$ denote the hyperbolic ball in $\mathbb{R}^{n+1}_+$. We denote the $L^2(\mathbb{R}^n)$-norm preserving translation and dilation of $f$ by 
$$
    f_z:= t^{-n/2}f\bigg( \frac{\cdot - x}{t}\bigg)
$$
for $z=(x,t) \in \mathbb{R}^{n+1}_+$. 

Fix a real-valued mother wavelet $\psi \in \mathcal{D}(\mathbb{R}^n)$ satisfying  
\[
    \text{supp}\,\psi \subseteq B(0,1), \quad  \int_{\mathbb{R}^n} \psi(x)\,dx =0, \quad  \text{and}\quad
\int_0^{\infty}\frac{|\widehat{\psi}(\xi t)|^2}{t}\,dt=1
\]
for $\xi \neq 0$. Under these conditions, $\{\psi_z\}_{z\in\mathbb{R}^{n+1}_+}$ is a continuous Parseval frame for $L^2(\mathbb{R}^n)$, i.e., \looseness=-1
$$
    \|f\|_{L^2(\mathbb{R}^n)}^2 = \int_{\mathbb{R}^{n+1}_+} |\langle f,\psi_z\rangle|^2 \,d\lambda(z)
$$
for all $f \in L^2(\mathbb{R}^n)$. 
The same idea leads to the following result in the product setting. Below, we write $\psi_{z} := \psi_{z_1}\otimes \psi_{z_2}$ for $z = (z_1,z_2) \in \mathbb{R}^{n_1+1}_+ \times \mathbb{R}^{n_2+1}_+$.
\begin{lemma}
    The set $\{\psi_{z}\}_{z \in \mathbb{R}^{n_1+1}_+ \times \mathbb{R}^{n_2+1}_+}$ is a continuous Parseval frame for $L^2(\mathbb{R}^{n_1}\times\mathbb{R}^{n_2})$.
\end{lemma}
\begin{proof}
    Let $f \in L^2(\R^{n_1} \times \R^{n_2})$. Denoting $\widetilde{\psi}_t(x) = t_1^{-n_1/2}\psi\big(-\frac{x_1}{t_1}\big)t_2^{-n_2/2}\psi\big(-\frac{x_2}{t_2}\big)$, applying Fubini's theorem, Plancherel's identity, the fact $\widehat{\widetilde{\psi}_t}(\xi) = (\widehat{\widetilde{\psi}})_{t^{-1}}(\xi)$, and the assumption on $\psi$, we have 
    \begin{align*}
        \int_{\R_+^{n_1 + 1}} \int_{\R_+^{n_2 + 1}} | &\langle f, \psi_{z} \rangle |^2 \, d\lambda(z) = \int_0^\infty \int_{\R^{n_1}} \int_0^\infty \int_{\R^{n_2}} |f * \tilde{\psi}_{t}(x)|^2 \, dx_2 \frac{dt_2}{t_2^{n_2+1}} dx_1 \frac{dt_1}{t_1^{n_1+1}}\\
        &= \int_0^\infty  \int_0^\infty \int_{\R^{n_1}} \int_{\R^{n_2}} |\widehat{f}(\xi)|^2 |\widehat{\widetilde{\psi}_{t}}(\xi)|^2  \, d\xi_2  d\xi_1 \frac{dt_2}{t_2^{n_2+1}} \frac{dt_1}{t_1^{n_1+1}}\\
        &= \int_{\R^{n_1}} \int_{\R^{n_2}} |\widehat{f}(\xi)|^2 \int_0^\infty\frac{|\widehat{\psi}(-\xi_1t_1)|^2}{t_1}\,dt_1\int_0^\infty\frac{|\widehat{\psi}(-\xi_2t_2)|^2}{t_2}\,dt_2\,d\xi_2d\xi_1\\
        &= \|\widehat{f}\|_{L^2(\R^{n_1} \times \R^{n_2})}^2\\ 
        &= \|f\|_{L^2(\R^{n_1} \times \R^{n_2})}^2,
    \end{align*}
    as desired.
\end{proof}

Recall that an operator $T\colon \DD(\mathbb{R}^n)\rightarrow\DD(\mathbb{R}^n)'$ satisfies the weak boundedness property if 
$$
    \sup_{z \in \mathbb{R}^{n+1}_+}\sup_{f,g \in \BB} |\langle Tf_{z}, g_{z}\rangle| < \infty
$$
for every bounded $\BB \subseteq \mathcal{D}(\mathbb{R}^n)$, and $T$ satisfies the weak compactness property if 
$$
    \lim_{z\rightarrow \infty}\sup_{f,g \in \mathcal{B}}|\langle Tf_{z},g_{z}\rangle|=0
$$
for every bounded $\mathcal{B} \subseteq \DD(\mathbb{R}^n)$, where $\displaystyle\lim_{z\rightarrow \infty}$ means $\displaystyle \lim_{R\rightarrow \infty} \sup_{z \in D((0,1),R)^c}$. While these properties are clearly necessary for the $L^2(\mathbb{R}^n)$-boundedness/compactness of $T$ by Cauchy-Schwarz, they are also sufficient for the boundedness/compactness of cancellative SIOs, i.e., SIOs $T$ such that $T1=T^t1=0$. As observed in \cite{MS2023}, the $L^2(\mathbb{R}^n)$-compactness only requires the following weaker version of the weak compactness property: \looseness=-1
$$
    \lim_{z\rightarrow \infty}\sup_{w\in D(z,R)}|\langle T\psi_{z},\psi_{w}\rangle|=0.
$$
for every $R>0$. We extend this formulation to the bi-parameter setting as follows. 
\begin{defin}\label{BiparameterWeakCompactnessDefinition}
    An operator $T\colon \mathcal{D}(\mathbb{R}^{n_1})\otimes\mathcal{D}(\mathbb{R}^{n_2})\rightarrow (\mathcal{D}(\mathbb{R}^{n_1})\otimes\mathcal{D}(\mathbb{R}^{n_2}))'$ satisfies the product weak compactness property if 
    $$
        \lim_{z\rightarrow \infty}\sup_{w \in D(z_1,R)\times D(z_2,R)}|\langle T\psi_{z},\psi_{w}\rangle| = 0
    $$
    for every $R>0$. 
\end{defin}
\noindent Observe that product weak compactness is invariant under full and partial transposes. 
\begin{prop}\label{productWCnecessary}
    If a linear operator is compact on $L^2(\mathbb{R}^{n_1}\times\mathbb{R}^{n_2})$, then it satisfies the product weak compactness property.
\end{prop}
\begin{proof}
    Let $T$ be compact on $L^2(\mathbb{R}^{n_1}\times\mathbb{R}^{n_2})$. Since $\psi_z \rightarrow 0$ weakly in $L^2(\R^{n_1} \times \R^{n_2})$ as $z \rightarrow \infty$, the compactness of $T$ implies that 
    \[
        \limsup_{z\rightarrow \infty}\sup_{w \in D(z_1,R)\times D(z_2,R)}|\langle T\psi_{z},\psi_{w}\rangle| \lesssim \limsup_{z \rightarrow \infty} \|T \psi_z \|_{L^2(\R^{n_1} \times \R^{n_2})} = 0
    \]
    for any $R > 0$.
\end{proof}

In addressing compactness, the role of $\text{BMO}(\R^n)$ is played by its subspace $\text{CMO}(\mathbb{R}^n)$, which was originally defined to be the closure of the continuous functions that vanish at infinity, $C_0(\mathbb{R}^n)$, in $\text{BMO}(\mathbb{R}^n)$ by Coifman and Weiss in \cite{CW1977}*{p. 638}. We extend this definition to the bi-parameter setting using the Chang-Fefferman product $\text{BMO}(\mathbb{R}^{n_1}\times\mathbb{R}^{n_2})$ of \cite{CF1980}.
\begin{defin}\label{CMODefinition}
    The space $\text{CMO}(\mathbb{R}^{n_1}\times\mathbb{R}^{n_2})$ is the closure of $C_0(\mathbb{R}^{n_1}\times\mathbb{R}^{n_2})$ in the space $\text{BMO}(\mathbb{R}^{n_1}\times\mathbb{R}^{n_2})$. 
\end{defin}

We next introduce the mixed weak compactness/CMO property, which is motivated by the weak boundedness/BMO condition introduced by Ou in \cite{O2017}.\looseness=-1
\begin{defin}
        A bi-parameter CZO $T$ satisfies the mixed weak compactness/CMO property if $\langle T(1 \otimes \psi_z), \cdot \otimes \psi_w \rangle \in \text{CMO}(\R^{n_1})$ for every $z,w \in \R_+^{n_2+1}$, 
         \[
            \lim_{z \rightarrow \infty} \sup_{w \in D(z,M)} \|\langle T(1 \otimes \psi_z) , \cdot \otimes \psi_{w} \rangle \|_{\text{BMO}(\R^{n_1})} = 0
         \]
        for every $M \geq 0$, $\langle T(\psi_z \otimes 1 ), \psi_w \otimes \cdot \rangle \in \text{CMO}(\R^{n_2})$ for every $z,w \in \R_+^{n_1+1}$, 
         \[
            \lim_{z \rightarrow \infty} \sup_{w \in D(z,M)} \|\langle T(\psi_z \otimes 1 ), \psi_{w} \otimes \cdot \rangle\|_{\text{BMO}(\R^{n_2})} = 0
         \]
         for every $M \ge 0$, and the same four conditions all hold with $T^t$ in place of $T$.
    \end{defin}
    \noindent Observe that the mixed weak compactness/CMO property is invariant under full and partial transposes.

    \begin{prop}\label{WC/CMOnecessary}
        If a bi-parameter CZO $T$ is compact on $L^2(\mathbb{R}^{n_1}\times\mathbb{R}^{n_2})$, then it satisfies the mixed weak compactness/CMO property.
    \end{prop}
    \begin{proof} 
    Since the operator $T_{z,w}$ given by $T_{z,w}f=\langle T(f \otimes \psi_z), \cdot \otimes \psi_w \rangle$ is a compact one-parameter CZO, we have that $\langle T(1 \otimes \psi_z), \cdot \otimes \psi_w \rangle \in \text{CMO}(\R^{n_1})$ by Theorem B. 
    
    For the second condition, note that $\{T_{z,w}\}_{z,w \in \mathbb{R}^{n_2+1}_+}$ is a bounded subset of $\delta \text{CZO}(\mathbb{R}^{n_1})$ as the $T_{z,w}$ are clearly uniformly bounded on $L^2(\mathbb{R}^{n_1})$ and the kernels of $T_{z,w}$, which are given by $\langle K(x,y) \psi_z, \psi_w \rangle$ for a $\delta$CZO-valued kernel $K$, satisfy uniform size and smoothness estimates in $z,w \in \mathbb{R}^{n_2+1}_+$. Therefore we have by the estimate of \cite{AM1986} that \looseness=-1
    $$
        \sup_{z,w \in \mathbb{R}^{n_2+1}_+} \|(T_{z,w})^t\|_{H^p(\R^{n_1}) \rightarrow L^p(\R^{n_1})} \lesssim 1
    $$ 
    for $\frac{n_1}{n_1+\delta}<p<1$, where $H^p(\mathbb{R}^{n_1})$ denotes the classical real Hardy space. We also have that 
    $$
        \sup_{w \in \mathbb{R}^{n_2+1}_+} \|(T_{z,w})^t\|_{L^2(\R^{n_1}) \rightarrow L^2(\R^{n_1})} \leq V(z),
    $$
    where $V(z)\rightarrow 0$ as $z \rightarrow \infty$, since if there exist $z_n \rightarrow \infty$ and $w_n \in \mathbb{R}^{n_2+1}_+$ such that $\|(T_{z_n, w_n})^t\|_{L^2(\R^{n_1}) \rightarrow L^2(\R^{n_1})} \geq \delta > 0$, then there are $L^2(\mathbb{R}^{n_1})$-normalized $f_n$ and $g_n$ satisfying $|\langle T(f_n \otimes \psi_{z_n}), g_n \otimes \psi_{w_n} \rangle| \geq \frac{\delta}{2}$, but then $f_n \otimes \psi_{z_n} \rightarrow 0$ weakly, violating the compactness of $T$. \looseness=-1
    
    We now claim that for $p \in (\frac{n_1}{n_1+\delta}, 1)$ and $\theta \in (0,1)$ satisfying $p \theta + 2(1-\theta) = 1$, we have
    $$
        \sup_{w\in\mathbb{R}^{n_2+1}_+} \|(T_{z,w})^t\|_{H^1(\R^{n_1}) \rightarrow L^1(\R^{n_1})} \lesssim V(z)^{2(1-\theta)} =: V'(z).
    $$ 
    Indeed, if $a$ is an $L^\infty(\mathbb{R}^{n_1})$ atom for $H^1(\mathbb{R}^{n_1})$, meaning $\text{supp}\, a \subseteq Q$, $\int a = 0$, and $\|a\|_{L^{\infty}(\mathbb{R}^{n_1})} \leq |Q|^{-1}$ for a cube $Q$, then $\| a \|_{H^p(\mathbb{R}^{n_1})} \lesssim |Q|^{\frac{1}{p} - 1}$ and $\|a\|_{L^2(\mathbb{R}^{n_1})} \leq |Q|^{-\frac{1}{2}}$. Considering the holomorphic function $\zeta \mapsto \int_{\R^n} |(T_{z,w})^ta(x)|^{\zeta}\,dx$, we find the estimates 
    $$
        \left| \int_{\R^n} |(T_{z,w})^ta(x)|^{\zeta}\,dx \right| \lesssim \|(T_{z,w})^t\|_{H^p(\mathbb{R}^{n_1}) \rightarrow L^p(\mathbb{R}^{n_1})}^p |Q|^{1-p} \lesssim |Q|^{1-p}
    $$
    and 
    $$
        \left| \int_{\R^n} |(T_{z,w})^ta(x)|^{\zeta}\,dx \right| \lesssim \|(T_{z,w})^t\|_{L^2(\mathbb{R}^{n_1}) \rightarrow L^2(\mathbb{R}^{n_1})}^2 |Q|^{-1} \lesssim V(z)^2 |Q|^{-1}
    $$ 
    when $\Re \zeta = p$ and $\Re \zeta = 2$, respectively. By the Hadamard three-lines lemma, we get
    $$
        \int_{\R^n} |(T_{z,w})^ta(x)|\,dx \lesssim |Q|^{(1-p)\theta} |Q|^{\theta-1} V(z)^{2(1-\theta)} = V(z)^{2(1-\theta)},
    $$
    which implies the claimed estimate by the atomic decomposition of $H^1(\mathbb{R}^{n_1})$. 
    
    Now, by $H^1$-BMO duality, we have 
    $$
        \sup_{w\in\mathbb{R}^{n_2+1}_+} \|T_{z,w}\|_{L^\infty(\R^{n_1}) \rightarrow \text{BMO}(\R^{n_1})} \lesssim V'(z).
    $$ In particular,
    \[
        \sup_{w\in\mathbb{R}^{n_2+1}_+}\| \langle T(1 \otimes \psi_z), \cdot \otimes \psi_w\rangle\|_{\text{BMO}(\R^{n_1})} = \sup_{w\in\mathbb{R}^{n_2+1}_+} \|T_{z,w} 1\|_{\text{BMO}(\R^{n_1})} \lesssim V'(z),
    \]
    and therefore
    \[
        \lim_{z \rightarrow \infty} \sup_{w \in D(z,M)} \|\langle T(1 \otimes \psi_z) , \cdot \otimes \psi_{w} \rangle \|_{\text{BMO}(\R^{n_1})} = 0.
    \]
    
    The other conditions on the form $\langle T(\psi_z \otimes 1 ), \psi_{w} \otimes \cdot \rangle$ are handled similarly, and the invariance of compactness under full transposes give the analogous conditions for $T^t$.
    \end{proof}

\subsection{Comparison of Theorem \ref{bi-parameterCZOCompactnessTheorem} and Theorem C}\label{ConditionComparisonSection}

The main appeal to our argument for Theorem \ref{bi-parameterCZOCompactnessTheorem} lies in recognizing the decomposition of a bi-parameter CZO into top-level paraproducts, partial paraproducts, and a fully cancellative bi-parameter CZO. This perspective enables us to apply our new abstract formalism for partially localized operators, Theorem \ref{Abstractbi-parameterCompactnessTheorem}, to the partial paraproducts and the fully cancellative part of the bi-parameter CZO, while we handle the top-level paraproducts separately. We obtain the partial localization of the partial paraproducts and fully cancellative bi-parameter CZO from well-known one-parameter estimates through Lemma \ref{VarpiCoefficientsEstimate} and Lemma \ref{FullyCancellativeCoefficientEstimate} below. We contrast this approach with the proof of Theorem C from \cite{CYY2024}, which is based on directly applying hard estimates involving dyadic rectangles. Further, the argument for Theorem C requires the formulation of bi-parameter SIOs in terms of mixed-type scalar conditions on bi-parameter kernels, whereas we use Journ\'e's original vector-valued definition. 

In addition to the fact that we obtain the necessity of conditions in Theorem \ref{bi-parameterCZOCompactnessTheorem}, the first major difference between the statements of Theorem \ref{bi-parameterCZOCompactnessTheorem} and Theorem C is that we place no additional vanishing assumptions on the kernels of $T$ in the form of compact full and partial kernel representations. On the other hand, we note that the product weak compactness property of Theorem C is a nonsmooth and noncancellative version of the diagonal condition 
\begin{align*}
    \lim_{z\rightarrow \infty}|\langle T\psi_z,\psi_z\rangle| =0,
\end{align*}
which is clearly implied by our product weak compactness of Definition \ref{BiparameterWeakCompactnessDefinition} as our condition includes both on-diagonal and near-diagonal matrix coefficients $\langle T\psi_z,\psi_w\rangle$. 

The second main difference between the statements of Theorem \ref{bi-parameterCZOCompactnessTheorem} and Theorem C is that our mixed weak compactness/CMO condition involves vanishing estimates on terms of the form $\| \langle T(1 \otimes \psi_z), \cdot \otimes \psi_w \rangle\|_{\text{BMO}(\mathbb{R}^{n_1})}$ instead of on those of the form
\begin{align}\label{ComparisonDisplay}
    \sup_{a} |\langle T(1_Q \otimes 1_R), a \otimes 1_R \rangle| \approx \| \langle T(1_Q \otimes 1_R), \cdot \otimes 1_R \rangle \|_{L^1(Q) / \text{constants}},
\end{align}
where $Q\subseteq \mathbb{R}^{n_1}$ and $R\subseteq\mathbb{R}^{n_2}$ are cubes and the supremum is taken over functions $a$ on $\mathbb{R}^{n_1}$ such that $\text{supp}\,a \subseteq Q$, $\|a\|_{L^\infty(\mathbb{R}^{n_1})} \leq 1$, and $\int a = 0$ as in the diagonal CMO property of Theorem C. We illustrate how additional compact partial kernel assumptions, along with assumptions similar to \eqref{ComparisonDisplay}, imply the vanishing estimates in the mixed weak compactness/CMO property. Let \looseness=-1
$$
    \text{osc}_\lambda(f;Q) := \inf_{c \in \C} \inf_{\substack{E \subseteq Q\\ |E| \leq \lambda |Q|}} \|(f-c)1_{Q \setminus E}\|_{L^\infty(\mathbb{R}^n)},
$$
and note that 
$$
    \text{osc}_\lambda (f;Q) \leq \lambda^{-\frac{1}{p}} \inf_{c \in \C} |Q|^{-\frac{1}{p}} \|f-c\|_{L^{p,\infty}(Q)}
$$
for any $p \in (0,\infty)$, where $\|f\|_{L^{p,\infty}(Q)}:= \sup_{\lambda>0} \lambda|\{x \in Q\colon |f(x)|>\lambda\}|^{1/p}$ is the usual weak-$L^p$ quasi-norm. In particular, this notion of oscillation is smaller than the $L^1(Q)/\text{constants}$ norm used in \eqref{ComparisonDisplay}. See \cite{HNVW23}*{Chapter 11} for a general treatment of oscillation.\looseness=-1

\begin{prop}\label{ComparisonProp}
    Let  $T \in \delta \text{CZO}(\R^n)$. If for each cube $Q$ there exists a function $\phi_Q$ such that $\|\phi_Q\|_{L^\infty(\mathbb{R}^n)} \leq 1$, $\phi_Q 1_Q \equiv 1_Q$, and
    $$
        \sup_{Q} \text{osc}_{2^{-3-n}} (T \phi_Q; Q) \leq A,
    $$
    then
    $$
        \|T 1\|_{\text{BMO}(\R^n)} \lesssim
        A+\|K\|_{\delta}.
    $$
\end{prop}
\begin{proof}
    Put $\lambda = 2^{-3-n}$. Let $Q$ be a cube in $\R^{n}$ and $\mathcal{F}$ be the collection of the $3^n - 1$ adjacent cubes to $Q$. Then, since $\phi_Q 1_Q \equiv 1_Q$, we have 
    $$
        1 = \phi_Q + (1-\phi_Q) = \phi_Q + 1_{(3Q)^c}(1-\phi_Q) + \sum_{Q' \in \mathcal{F}} 1_{Q'} (1-\phi_Q),
    $$
    and so we can estimate
    \begin{align*}
        &\text{osc}_{2\lambda } (T 1; Q) \leq  \text{osc}_\lambda (T \phi_Q; Q) + \text{osc}_\lambda (T(1-\phi_Q); Q) \\
        & \quad\quad\lesssim A + |Q|^{-1} \|T (1_{(3Q)^c} (1-\phi_Q)) \|_{L^1(Q) / \text{constants}} + \sum_{Q' \in \mathcal{F}} |Q|^{-1} \|T(1_{Q'} (1-\phi_Q))\|_{L^1(Q)}.
    \end{align*}
    Since $\text{supp } (1_{(3Q)^c} (1-\phi_{Q})) \subseteq (3Q)^c$, we have from standard estimates that
    \begin{align*}
        |Q|^{-1} \|T (1_{(3Q)^c} (1-\phi_{Q}))\|_{L^1(Q) / \text{constants}} \lesssim \|K\|_{\delta} \|1-\phi_{Q}\|_{L^\infty(\mathbb{R}^n)}  \lesssim \|K\|_{\delta}.
    \end{align*}
    For the terms involving adjacent cubes $Q'$, we see
    \begin{align*}
        |Q|^{-1}\|T (1_{Q'}(1-\phi_Q))\|_{L^1(Q)} \leq |Q|^{-1}\|1 - \phi_Q\|_{L^\infty(\mathbb{R}^n)} \int_Q \int_{Q'} |K(x,y)| \, dydx \lesssim \|K\|_{\delta}.
    \end{align*}
    By \cite{HNVW23}*{Proposition 11.1.24}, we conclude 
    $$
        \|T1\|_{\text{BMO}(\R^n)} \lesssim \sup_Q \text{osc}_{2 \lambda} (T1; Q) \lesssim A+\|K\|_{\delta},
    $$
    as desired.
\end{proof}

Proposition \ref{ComparisonProp} immediately gives the following bound for bi-parameter CZOs. 
\begin{cor}\label{ComparisonCorollary}
    Let $T$ be a bi-parameter $\delta \text{CZO}$ such that $\langle K_1(x,y) \psi_z, \psi_w \rangle$ has kernel constants bounded by $V(z,w)$. If for each cube $Q$ and $z,w \in \mathbb{R}_+^{n_2+1}$ there exists a function $\phi_{Q,z,w}$ on $\mathbb{R}^{n_1}$ such that $\|\phi_{Q,z,w}\|_{L^{\infty}(\mathbb{R}^{n_1})} \leq 1$, $\phi_{Q,z,w} 1_Q \equiv 1_Q$, and
    $$
        \sup_Q \text{osc}_{2^{-3-n_1}} (\langle T(\phi_{Q,z,w} \otimes \psi_z), \cdot \otimes \psi_w \rangle; Q) \leq V(z,w),
    $$
    then
    $$
        \|\langle T(1 \otimes \psi_z), \cdot \otimes \psi_w \rangle\|_{\text{BMO}(\R^{n_1})} \lesssim V(z,w).
    $$
\end{cor}
\noindent Corollary \ref{ComparisonCorollary} gives part of the mixed weak compactness/CMO property when $V$ satisfies $\lim_{z \rightarrow \infty} \sup_{w \in D(z,M)} V(z,w) = 0$, in particular, when an expression analogous to \eqref{ComparisonDisplay} vanishes. The other mixed weak compactness/CMO estimates follow from similar hypotheses.\looseness=-1


\section{Endpoint Compactness and necessity of $T1, T^t1 \in \text{CMO}(\mathbb{R}^{n_1}\times \mathbb{R}^{n_2})$}\label{EndpointSection}

We next prove Theorem \ref{EndpointCompactness} and establish the necessity of $T1, T^t1 \in \text{CMO}(\R^{n_1} \times \R^{n_2})$. 
    \begin{lemma}\label{HPInerpolation}
        If an operator $T$ is bounded on $L^2(\R^{n_1} \times \R^{n_2})$ and from $H^p(\R^{n_1} \times \R^{n_2})$ to $L^p(\R^{n_1} \times \R^{n_2})$ for large enough $p< 1$, and $\theta \in (0,1)$ satisfies $\theta p + (1-\theta)2 = 1$, then $T$ is bounded from $H^1(\R^{n_1} \times \R^{n_2})$ to $L^1(\R^{n_1} \times \R^{n_2})$ with 
        $$
            \|T\|_{H^1(\R^{n_1} \times \R^{n_2}) \rightarrow L^1(\R^{n_1} \times \R^{n_2})} \lesssim \|T\|_{H^p(\R^{n_1} \times \R^{n_2}) \rightarrow L^p(\R^{n_1} \times \R^{n_2})}^{\theta p}\|T\|_{L^2(\R^{n_1} \times \R^{n_2}) \rightarrow L^2(\R^{n_1} \times \R^{n_2})}^{(1-\theta)2}.
        $$
    \end{lemma}
    \begin{proof}
        We use the formulation of the atomic decomposition characterization of $H^p(\R^{n_1} \times \R^{n_2})$ in \cite{HAN20102834}*{Theorem B}. Let $a$ be an atom for $H^1$ on an open set $\Omega \subseteq \R^{n_1} \times \R^{n_2}$. Since $p$ is sufficiently large, $|\Omega|^{1-\frac{1}{p}} a$ is an $H^p$ atom and, by definition, $\|a\|_{L^2(\mathbb{R}^{n_1}\times\mathbb{R}^{n_2})} \leq |\Omega|^{-1/2}$. The result follows by considering the holomorphic function
        $$
            \zeta \mapsto \int_{\R^{n_1} \times \R^{n_2}} |Ta(x)|^\zeta\,dx
        $$
        and repeating the argument from Proposition \ref{WC/CMOnecessary}.
    \end{proof}

We denote the standard dyadic grid on $\R^n$ by $\D$ and, for $N \in \mathbb{N}$, define the finite collection 
$\mathcal{D}_N := \big\{Q \in \mathcal{D}\colon Q \subseteq [-2^N,2^N)^n \text{ and } |Q| \ge 2^{-N}\big\}$. Given two $C^1$ compactly supported orthonormal wavelet bases for $L^2(\R^{n_1})$ and $L^2(\R^{n_2})$, $\{ \psi_Q\}_{Q \in \D}$ and $\{ \psi_S \}_{S \in \D}$, we define the finite rank projection operators $P_N^1$ and $P_N^2$ by
$$
    P_N^1 f = \sum_{Q \in \D_N} \langle f, \psi_Q \rangle \psi_Q \quad\text{and}\quad
    P_N^2 f = \sum_{S \in \D_N} \langle f, \psi_S \rangle \psi_S,
$$
where we have suppressed a finite sum over the auxiliary variables $\gamma \in \{ 1, \ldots, 2^{n_1}-1\}$ and $\delta \in \{1, \ldots, 2^{n_2} - 1 \}$ in the indexing of the wavelet bases. Define $P_N := P_N^1 \otimes P_N^2$.
    \begin{lemma}\label{Unconditionality}
        Each $P_N$ is a fully cancellative bi-parameter CZO with norm independent of $N$.\looseness=-1
    \end{lemma}
    \begin{proof}
        The fact that each $P_N^i$ is a one-parameter CZO with norm independent of $N$ is noted in \cite{M1992}*{Chapter 6} and follows by summing the kernels of the rank 1 terms and using support properties of the wavelets. Therefore, their tensor product is a bi-parameter CZO with norm independent of $N$. The fact that $P_N$ is fully cancellative follows directly from the mean-zero cancellation of each $\psi_Q$ and $\psi_S$.
    \end{proof}
    \begin{proof}[Proof of Theorem \ref{EndpointCompactness}]
        By compactness of $T$ on $L^2(\mathbb{R}^{n_1}\times\mathbb{R}^{n_2})$, we have that 
        $$
            \lim_{N\rightarrow\infty}\|T P_N^\perp \|_{L^2(\R^{n_1}\times\R^{n_2}) \rightarrow L^2(\R^{n_1}\times\R^{n_2})} = 0.
        $$
        By Lemma \ref{Unconditionality} and \cite{HAN20102834}*{Theorem 1}, $P_N$ and $P_N^\perp$ are uniformly bounded on $H^p(\R^{n_1}\times\R^{n_2})$ for $p \in \big(\max \{ \frac{n_1}{n_1 + 1} , \frac{n_2}{n_2 + 1} \}, 1\big]$. We also have that $T$ is bounded from $H^p(\R^{n_1}\times\R^{n_2})$ to $L^p(\R^{n_1}\times\R^{n_2})$ for sufficiently large $p < 1$ by \cite{MR0828217}, and so \looseness=-1
        $$
            \sup_{N\in\mathbb{N}}\|T P_N^\perp\|_{H^p(\R^{n_1} \times \R^{n_2}) \rightarrow L^p(\R^{n_1} \times \R^{n_2})} <\infty
        $$ 
        for such $p$. By Lemma \ref{HPInerpolation} we have $$
            \lim_{N\rightarrow\infty}\|T P_N^\perp \|_{H^1(\R^{n_1} \times \R^{n_2}) \rightarrow L^1(\R^{n_1} \times \R^{n_2})} = 0,
        $$
        which implies that $T P_N \rightarrow T$ uniformly as operators from $H^1(\R^{n_1} \times \R^{n_2})$ to $L^1(\R^{n_1} \times \R^{n_2})$; thus $T$ is compact from $H^1(\R^{n_1} \times \R^{n_2})$ to $L^1(\R^{n_1} \times \R^{n_2})$. Replacing $T$ with $T^*$ and taking the adjoint of $T^* P_N$, we get that $P_N T \rightarrow T$ in operator norm from $L^\infty(\R^{n_1} \times \R^{n_2})$ to $\text{BMO}(\R^{n_1} \times \R^{n_2})$, and since each $P_N$ is finite rank with range in $C_0(\R^{n_1} \times \R^{n_2})$, this gives that $T$ is compact from $L^\infty(\R^{n_1} \times \R^{n_2})$ to $\text{CMO}(\R^{n_1} \times \R^{n_2})$.
    \end{proof}

    \begin{cor}\label{nec3}
        If $T$ is an $L^2(\mathbb{R}^{n_1}\times\mathbb{R}^{n_2})$-compact bi-parameter CZO, then $T1, T^t1 \in$ \newline$\text{CMO}(\mathbb{R}^{n_1}\times\mathbb{R}^{n_2})$.
    \end{cor}


\section{Compactness of abstract partially localized operators}\label{AbstractSection}

    In this section, we introduce the notion of partially localized operators and characterize their compactness with the weak compactness property, which generalizes the notion from Definition \ref{BiparameterWeakCompactnessDefinition}. We develop this theory in an abstract setting of tensor products of Hilbert spaces, and later apply it in special cases to deduce the compactness of the bi-parameter partial paraproducts and fully cancellative bi-parameter CZOs satisfying the product weak compactness property. 
    
    Let $H_1$ and $H_2$ be separable Hilbert spaces and $G$ be a separable locally compact group with left Haar measure $dz$. We construct an appropriate exhaustion $\{E_N\}_{N\in\mathbb{N}}$ of $G$. Let $U$ be an open, precompact neighborhood of the identity, let $\{ p_N \}_{N \in \mathbb{N}}$ be a dense subset of $G$ with $p_1 = e$, and define $F_N := \bigcup_{n =1}^N \left( p_n U \cup U^{-1} p_n^{-1} \right)$. Note that $F_{N+1} \supseteq F_N \supseteq \{p_1,\ldots, p_N \}$ and $F_N^{-1}  =F_N$. Now define $E_N := F_N^N = F_N * F_N * \cdots * F_N$ ($N$ times). Then each $E_N$ is open and precompact. Also note that, since $F_N \subseteq F_{N+1}$ and $F_N \ni e$, we have $F_N \subseteq E_N \subseteq E_{N+1}$. Also, $E_N * E_M \subseteq E_{\max \{N,M\}}^2 \subseteq E_{2 \max \{N,M\}}$. Since $E_N^{-1} = E_N$, it follows that $E := \bigcup_{N \in \mathbb{N}} E_N$ is an open subgroup of $G$. Since $E^c$ is the union of its cosets, $E$ is also closed. Since $E$ contains a dense subset of $G$ by construction, it follows that $E=G$. The important properties of $\{ E_N \}_{N \in \mathbb{N}}$ are that $E_N^2 \subseteq E_{2N}$, $E_N \subseteq E_{N+1}$, and that each $E_N$ is open and precompact.
    
    Suppose that $\{\psi_z\}_{z \in G} \subseteq H_2$ is a bounded collection such that $z \mapsto \psi_z$ is norm continuous and that forms a continuous Parseval frame for $H_2$ in the sense that \looseness=-1
    \[
        \langle f,g\rangle_{H_2} = \int_G \langle f,\psi_z \rangle_{H_2}\langle \psi_z, g\rangle_{H_2} \, dz
    \]
    for all $f,g \in H_2$. 
    
    We denote by $\HS = H_1 \Ptens H_2$ the completion of the algebraic tensor product $H_1 \otimes H_2$ with respect to the inner product defined by the linear extension of 
    $$
        \langle a_1\otimes b_1 , a_2 \otimes b_2 \rangle_{H_1 \otimes H_2} := \langle a_1,a_2\rangle_{H_1}\langle b_1,b_2\rangle_{H_2}
    $$
    for $a_1,a_2 \in H_1$ and $b_1,b_2 \in H_2$. For $f = \sum_{n=1}^{N} c_n(a_n\otimes b_n) \in H_1 \otimes H_2$ and $g \in H_2$, we define $\langle f,g\rangle_{H_2} \in H_1$ by
    $$
        \langle f, g \rangle_{H_2} := \sum_{n=1}^N c_na_n \langle b_n, g \rangle_{H_2},
    $$
    and use the same notation to denote its natural well-defined extension to all $f \in \HS$. 
    
    We define the wavelet transform $\W \colon \HS \rightarrow L^2(G;H_1)$ by
    $$
        \W f(z) := \langle f, \psi_z \rangle_{H_2}.
    $$
    Note that $\W$ is an isometry since the scalar-valued wavelet transform $H_2 \ni f \mapsto \langle f,\psi_z\rangle_{H_2}$ is assumed to be an isometry on $L^2(G)$ -- this fact extends to the above Hilbert space-valued version by expanding into an orthonormal basis. Since $\W$ is an isometry, we have that 
    \[
        \langle f, g\rangle_{\HS}= \int_G \langle \W f(z), \W g(z) \rangle_{H_1} \, dz  = \int_G \langle \W f(z) \otimes \psi_z , g\rangle_{\HS} \, dz 
    \]
    for all $f, g \in \HS$. 
    
    Given a linear operator $T\colon \HS \rightarrow \HS$ and $z,w \in G$, we define the restricted operator $T_{(z,w)}\colon H_1 \rightarrow H_1$ for $h \in H_1$ by 
    $$
        T_{(z,w)}h := \langle T(h \otimes \psi_z), \psi_w \rangle_{H_2}.
    $$

    \begin{defin}\label{PartialLocalizationDefinition}
        A bounded linear operator $T \colon \HS \rightarrow \HS$ is partially localized if 
        \begin{align*}
            \sup_{z \in G} \omega(z)^{-1} &\int_G \|T_{(z,w)} \|_{H_1 \rightarrow H_1} \omega(w) \, dw < \infty,\\
            \sup_{w \in G} \omega(w)^{-1} &\int_G \|T_{(z,w)} \|_{H_1 \rightarrow H_1} \omega(z) \, dz < \infty,\\
            \lim_{N \rightarrow \infty} \sup_{z \in G} \omega(z)^{-1} &\int_{G \setminus zE_N} \|T_{(z,w)} \|_{H_1 \rightarrow H_1} \omega(w) \, dw=0, \quad\text{and}\\
            \lim_{N \rightarrow \infty} \sup_{w \in G} \omega(w)^{-1} &\int_{G \setminus wE_N} \|T_{(z,w)} \|_{H_1 \rightarrow H_1} \omega(z) \, dz=0
        \end{align*}
        for some continuous function $\omega \colon G \rightarrow \R_+$ satisfying $\omega(zw) \leq \omega(z) \omega(w)$ for all $z,w \in G$.
    \end{defin}
    \noindent Note that since each $E_N$ is open and precompact, the condition 
    $$
        \lim_{N \rightarrow \infty} \sup_{z \in G} \omega(z)^{-1} \int_{G \setminus zE_N} \|T_{(z,w)} \|_{H_1 \rightarrow H_1} \omega(w) \, dw=0
    $$ 
    is equivalent to the assertion that for every $\epsilon > 0$ there is a compact set $K \subseteq G$ with $$\sup_{z \in G} \omega(z)^{-1} \int_{G \setminus zK} \|T_{(z,w)} \|_{H_1 \rightarrow H_1} \omega(w) \, dw < \epsilon,$$ and similarly for the other vanishing condition, so that the definition of partial localization is independent of the choice of exhaustion $\{ E_N\}_{N \in \mathbb{N}}$. 
    
    We will use the following sufficient condition for partial localization.
    \begin{lemma}\label{MatrixBound>Localized}
        If $T \colon \HS \rightarrow \HS$ is a bounded linear operator such that 
        $$
            \|T_{(z,w)}\|_{H_1 \rightarrow H_1} \leq L(z^{-1} w)
        $$ 
        for all $z, w \in G$, where $L(\cdot^{-1}) = L(\cdot)$ and $L \in L^1(G, \omega(z) dz)$ for some continuous function $\omega \colon G \rightarrow \R_+$ satisfying $\omega(zw) \leq \omega(z)\omega(w)$ for all $z,w \in G$, then $T$ is partially localized.
    \end{lemma}
    \begin{proof}
        Let $S \subseteq G$ be measurable. Then
        $$
            \int_{S} \|T_{(z,w)}\|_{H_1 \rightarrow H_1} \omega(w) \, dw \leq \int_S L(z^{-1} w) \omega(w) \leq  \omega(z) \int_{z^{-1} S} L(w) \, \omega(w) \, dw.
        $$
        If $S = G$, this gives
        \[
            \sup_{z \in G} \omega(z)^{-1} \int_G \|T_{(z,w)} \|_{H_1 \rightarrow H_1} \omega(w) \, dw \leq \int_G L(w) \, \omega(w) \, dw < \infty.
        \]
        If $S = G \setminus zE_N$, then by the dominated convergence theorem 
        $$
            \lim_{N \rightarrow \infty} \sup_{z \in G} \omega(z)^{-1} \int_{G \setminus zE_N} \|T_{(z,w)} \|_{H_1 \rightarrow H_1} \omega(w) \, dw \leq \lim_{N \rightarrow \infty} \int_{G \setminus E_N} L(w) \omega(w) \, dw = 0.
        $$
        The other estimates follow from the same argument and the assumption $L(\cdot^{-1}) = L(\cdot)$.
    \end{proof}

    We next define weak compactness in this setting.
    \begin{defin}\label{WeakCptDefinition}
        A linear operator $T \colon \HS \rightarrow \HS$ is weakly compact if $T_{(z,w)}$ is compact on $H_1$ for all $z,w \in G$ and 
        \[
            \lim_{z \rightarrow \infty} \sup_{w \in z E_M} \|T_{(z,w)}\|_{H_1 \rightarrow H_1} = 0
        \]
        for every $M > 0$.
    \end{defin}
    \noindent Similarly to Definition \ref{PartialLocalizationDefinition}, we can equivalently replace $E_M$ with an arbitrary compact set $K \subseteq G$, so the definition does not depend on the exhaustion $\{ E_N \}_{N \in \mathbb{N}}$.
   
    Weak compactness is necessary for the compactness of $T$, so long as $\psi_z \rightarrow 0$ weakly. \looseness=-1
    \begin{prop}
        If a linear operator $T \colon \HS \rightarrow \HS$ is compact and $\psi_z \rightarrow 0$ weakly in $H_2$ as $z \rightarrow \infty$, then $T$ is weakly compact. 
    \end{prop}
    \begin{proof}
        The compactness of each $T_{(z,w)}$ is clear. For the other condition, suppose $f_n \rightarrow 0$ weakly in $H_2$ and $g_n$ is bounded in $H_2$. If $\limsup_{n \rightarrow \infty} \| \langle T(\cdot \otimes f_n), \cdot \otimes g_n \rangle \|_{H_1 \rightarrow H_1} > 0$, then there exist $a_n, b_n \in H_1$ such that $\|a_n\|_{H_1} = \|b_n\|_{H_1} = 1$ and
        $$
            \limsup_{n \rightarrow \infty} \|T(a_n \otimes f_n) \|_\HS \gtrsim  \limsup_{n \rightarrow \infty} | \langle T(a_n \otimes f_n), b_n \otimes g_n\rangle | > 0.
        $$
        But $a_n \otimes f_n \rightarrow 0$ weakly in $\HS$, so by compactness of $T$ this is a contradiction. To conclude, we use that $\psi_z \rightarrow 0$ weakly in $H_2$ and make the obvious modifications with the indexing set $G$ in place of sequences.
    \end{proof}

    We establish the following Riesz-Kolmogorov type criterion for the precompact subsets of $\HS$, which will be key to the proof of Theorem \ref{Abstractbi-parameterCompactnessTheorem}.
    \begin{prop}\label{RieszKolmogorov}
        A bounded set $\mathcal{K} \subseteq \HS$ is precompact if and only if $\{\W f(z) \colon f \in \mathcal{K}\} \subseteq H_1$ is precompact for each $z \in G$ and 
        \[
            \lim_{N \rightarrow \infty} \sup_{f \in \mathcal{K}} \int_{G \setminus E_N} \| \W f (z) \|_{H_1}^2 \, dz = 0.
        \]
    \end{prop}
    \begin{proof}
        For the forward direction, first note that each $\{\W f(z)\colon f \in \mathcal{K}\}$ is precompact in $H_1$ by the continuity of the map $f \mapsto \W f(z)$ from $\HS$ to $H_1$. Also, by the continuity of $\W \colon \HS \rightarrow L^2(G;H_1)$, we have that $\W (\mathcal{K})$ is precompact in $L^2(G;H_1)$, which implies
        \[
            \lim_{N \rightarrow \infty} \sup_{f \in \mathcal{K}} \int_{G \setminus E_N} \| \W f(z) \|_{H_1}^2 \, dz = 0.
        \]
        
        For the reverse direction, assume without loss of generality that $\mathcal{K}$ is closed. Suppose $\{ f_n \} \subseteq \mathcal{K}$ converges weakly to $f$.  By the first hypothesis, diagonalization, separability of $G$, and strong continuity of $z \mapsto \psi_z$, we can pass to a subsequence, which we also denote by $\{f_n\}$, such that
        \[
            \|\W (f_m - f_n)(z)\|_{H_1} \rightarrow 0
        \]
        for each $z \in G$ as $n,m \rightarrow \infty$. Since $f_n \rightarrow f$ weakly, we have that  $\W f_n(z) \rightarrow \W f(z)$ weakly, and since weak limits are unique, it follows that
        \[
            \|\W (f - f_n)(z)\|_{H_1} \rightarrow 0.
        \]
        Note that for any $M \in \N$,
        \[
            \| f - f_n \|_\HS^2 = \int_{E_M} \|\W (f-f_n)(z) \|_{H_1}^2 \, dz + \int_{G \setminus E_M} \|\W (f-f_n)(z) \|_{H_1}^2 \, dz .
        \]
        Now let $\epsilon > 0$ and choose $M = M(\epsilon) \in \mathbb{N}$ large enough so that 
        \[
            \sup_{g \in \mathcal{K}} \int_{G \setminus E_M} \| \W g (z) \|_{H_1}^2 \, d \mu(z) < \epsilon.
        \]
        Then
        \[
            \sup_{n \in \N} \int_{G \setminus E_M} \|\W (f-f_n)(z) \|_{H_1}^2 \, dz < 4\epsilon.
        \]
        Also note that
        \[
            \int_{E_M} \| \W(f-f_n)(z) \|_{H_1}^2 \, dz \rightarrow 0
        \]
        by dominated convergence, so we can choose $n$ large enough so that $\|f - f_n\|_\HS^2 < 5 \epsilon$. \looseness=-1
    \end{proof}

    We will use the following in the proof of Theorem \ref{Abstractbi-parameterCompactnessTheorem}.
    \begin{lemma}\label{WeakCpt>vanishing}
        If a bounded linear operator $T \colon \HS \rightarrow \HS$ is partially localized and weakly compact, then 
        \[
            \lim_{N \rightarrow \infty} \sup_{\substack{f \in \HS \\ \|f\|_{\HS} \leq 1}} \int_{G \setminus E_N} \|\W T f(z)\|_{H_1}^2 \, dz = 0.
        \]
    \end{lemma}
    \begin{proof}
        We first show that 
        \begin{align}\label{Lemma36Display}
            \lim_{N \rightarrow \infty} \sup_{z \in G} \omega(z)^{-1} \int_{G \setminus E_N} \|T_{(z,w)}\|_{H_1 \rightarrow H_1}  \omega(w) \, dw=0.
        \end{align}
        Partial localization allows us to just prove
        \[
            \lim_{N \rightarrow \infty} \sup_{z \in G} \omega(z)^{-1} \int_{(G \setminus E_N) \cap zE_M} \|T_{(z,w)}\|_{H_1 \rightarrow H_1} \, \omega(w) \, dw = 0
        \]
        for every $M \geq 0$. Let $N' \geq M$ and $N = 2N'$. If the domain of integration is nonempty, then $z \notin E_{N'}$. Indeed, otherwise $w \in E_{N'} E_M \subseteq E_{N'}^2 \subseteq E_{2N'} = E_N$. Therefore, letting $C(M) = \int_{E_M} \omega(w)\,dw$, we bound
        \[
            \sup_{z \in G} \omega(z)^{-1} \int_{(G\setminus E_N) \cap zE_M} \|T_{(z,w)}\|_{H_1 \rightarrow H_1} \, \omega(w) \, dw \leq C(M) \sup_{\substack{z \notin E_{N'}\\ w \in zE_M}} \|T_{(z,w)}\|_{H_1 \rightarrow H_1},
        \]
        which goes to zero with $N'$ by weak compactness, establishing \eqref{Lemma36Display}. 
        
        We now estimate
        \[
            \|\W T f(z)\|_{H_1} = \sup_{\|g\|_{H_1} \leq 1} \left| \langle T f, g \otimes \psi_{z} \rangle_\HS \right| 
        \]
        pointwise. For any such $g$, we have that
        \begin{align*}
            |\langle T f, g \otimes \psi_{z}  \rangle_\HS |^2 &= \left| \int_{G} \langle T(\W f(w) \otimes \psi_w) , g \otimes \psi_{z} \rangle_{\HS} \, dw \right|^2\\
            &\leq \left( \int_G \| \W f(w)\|_{H_1} \|T_{(w,z)} \|_{H_1\rightarrow H_1} \, dw \right)^2\\
            &\leq \omega(z)\int_{G} \| \W f(w) \|_{H_1}^2 \|T_{(w,z)}\|_{H_1 \rightarrow H_1} \omega(w)^{-1} \, dw\\
            &\quad\quad\times \omega(z)^{-1} \int_G \|T_{(w,z)}\|_{H_1 \rightarrow H_1} \omega(w) \, dw\\
            &\leq C\int_G \|\W f(w)\|_{H_1}^2 \omega(w)^{-1}\|T_{(w,z)}\|_{H_1 \rightarrow H_1} \omega(z) \, dw.
        \end{align*}
        The conclusion follows from Fubini's theorem and \eqref{Lemma36Display}, noting that the roles of $z$ and $w$ have been swapped.     
    \end{proof}

    \begin{proof}[Proof of Theorem \ref{Abstractbi-parameterCompactnessTheorem}]
        By Proposition \ref{RieszKolmogorov} and Lemma \ref{WeakCpt>vanishing}, it only remains to show that $\{ \W T f(z) \}_{\|f\|_\HS \leq 1}$ is compact in $H_1$ for each $z \in G$. To this end, let $P_N \colon H_1 \rightarrow H_1$ be a sequence of finite rank orthogonal projections increasing to the identity. We will show that
        \[
            \lim_{N \rightarrow \infty} \sup_{\|f\|_\HS \leq 1} \| P_N^\perp \W T f (z)\|_{H_1} = 0
        \]
        for all $z \in G$, which is equivalent to the precompactness of $\{\mathcal{W}Tf(z)\}_{\|f\|_{\mathcal{H}\leq 1}}$ by \cite{MSWW2023}*{Theorem B}. \looseness=-1
        
        By the same argument as in the final estimate in the proof of Lemma \ref{WeakCpt>vanishing}, we have
        \[
            \|P_N^\perp \W T f(z)\|_{H_1}^2 \lesssim \int_G \| \W f(w)\|_{H_1}^2 \omega(w)^{-1}\|P_N^\perp T_{(w,z)}\|_{H_1 \rightarrow H_1} \omega(z) \, dw.
        \]
        Note that $P_N^\perp T_{(w,\zeta)} \rightarrow P_N^\perp T_{(w,z)}$ in operator norm as $\zeta \rightarrow z$ uniformly in $w$ and $N$, since $z \mapsto \psi_z$ is norm continuous and $\{\psi_w\}_{w \in G}$ is bounded. This, along with the continuity of $\omega$, allow us to choose an open, precompact set $U = U(z) \ni z$ such that
        \begin{align*}
            \omega(w)^{-1}\|P_N^\perp T_{(w,z)}\|_{H_1 \rightarrow H_1} \omega(z) &\lesssim \omega(w)^{-1} \inf_{\zeta \in U} \|P_N^\perp T_{(w,\zeta)}\|_{H_1 \rightarrow H_1} \omega(\zeta)\\
            &\lesssim \omega(w)^{-1} \int_{U} \|P_N^\perp T_{(w,\zeta)}\|_{H_1 \rightarrow H_1} \omega(\zeta) \, d\zeta.
        \end{align*}
        It suffices to show that
        \[
            \lim_{N \rightarrow \infty} \sup_{\|f\|_\HS \leq 1} \int_{G} \| \W f(w)\|_{H_1}^2 \omega(w)^{-1} \int_U \|P_N^\perp T_{(w,\zeta)}\|_{H_1 \rightarrow H_1} \omega(\zeta) \, d\zeta dw = 0.
        \]
        
        Let $\epsilon > 0$, $\|f\|_\HS \leq 1$, and choose $M$ large enough so that
        \[
            \sup_{w \in G} \omega(w)^{-1} \int_{G \setminus wE_{M}} \|T_{(w,\zeta)} \|_{H_1 \rightarrow H_1} \omega(\zeta) \, d \zeta < \epsilon.
        \]
        Then
        \begin{align*}
            &\int_{G} \| \W f(w)\|_{H_1}^2 \omega(w)^{-1} \int_U \|P_N^\perp T_{(w,\zeta)}\|_{H_1 \rightarrow H_1} \omega(\zeta) \, d\zeta \, dw\\
            &\quad\quad\leq \epsilon+ \int_{G} \| \W f(w)\|_{H_1}^2 \omega(w)^{-1} \int_{U \cap w E_M} \|P_N^\perp T_{(w,\zeta)}\|_{H_1 \rightarrow H_1} \omega(\zeta) \, d\zeta dw.
        \end{align*}
        Note that the integrand in $w$ has support in some $E_{M'}$. Indeed, since $U$ is precompact, $U \subseteq E_{M''}$ for some $M''$. If $\zeta \in U \cap w E_M \subseteq E_{M''} \cap w E_M$, then $w \in E_{M''} E_M^{-1} \subseteq E_{M''} E_{M'''} \subseteq E_{2\max (M'',M''')} =: E_{M'}$. Therefore, we can estimate
        \begin{align*}
            \int_{G} \| \W f(w)\|_{H_1}^2 &\omega(w)^{-1} \int_{U \cap w E_M} \|P_N^\perp T_{(w,\zeta)}\|_{H_1 \rightarrow H_1} \omega(\zeta) \, d\zeta dw\\
            &\leq \int_{E_{M'}} \| \W f(w) \|_{H_1}^2 \omega(w)^{-1} \int_{U} \|P_N^\perp T_{(w,\zeta)}\|_{H_1 \rightarrow H_1} \omega(\zeta) \, d\zeta dw\\
            &\leq C \int_{E_{M'}} \int_U \|P_N^\perp T_{(w,\zeta)} \|_{H_1 \rightarrow H_1} \, d\zeta dw,
        \end{align*}
        where $C = \|\omega^{-1}\|_{L^\infty(E_{M'})} \|\omega\|_{L^\infty(U)} \| \|\psi_z\|_{H_2}^2 \|_{L^\infty(E_{M'})} < \infty$. We can then choose $N$ such that the above is bounded by $\epsilon$ by the compactness of each $T_{(w,\zeta)}$ and dominated convergence.
    \end{proof}


\section{Compactness of Bi-Parameter Paraproducts}\label{ParaproductSection}

In this section, we address the compactness of the bi-parameter paraproducts. We first introduce the top-level paraproducts and show that they are compact using the hypothesis $T1,T^t1,T_t1,T^t_t1 \in \text{CMO}(\R^{n_1}\times\R^{n_2})$. We then define the partial paraproducts and deduce their compactness from Theorem \ref{Abstractbi-parameterCompactnessTheorem} and the mixed weak compactness/CMO property.

\subsection{Top-level paraproducts}

Below, we let $\phi_{z}^1 := \varphi_{z_1}^1 \otimes \varphi_{z_2}^1$, where $\varphi_{z_i}^1 = t_i^{-n_i} \varphi\big( \frac{\cdot - x_i}{t_i} \big)$ for $z_i = (x_i,t_i) \in \mathbb{R}^{n_i+1}_+$  and some fixed nonnegative $L^1(\mathbb{R}^{n_i})$-normalized $\varphi \in \DD(\R^{n_i})$.

\begin{defin}\label{TopLevelParaproductDefinition}
    Given $b \in \text{BMO}(\R^{n_1} \times \R^{n_2})$, the top-level paraproduct, $\Pi_b$, is given by 
    \[
        \Pi_b f = \iint_{\R_+^{n_1+1} \times \R_+^{n_2 + 1}} \langle b, \psi_z \rangle \langle f, \phi_z^1 \rangle \psi_z \, dz.
    \]
\end{defin}
\noindent Note that 
$$\Pi_b 1 = b \quad\quad\text{and}\quad\quad \Pi_{b,t}1 = \Pi_{b,t}^t 1 = \Pi_b^t1 = 0.$$

We recall that top-level paraproducts are bi-parameter CZOs, see \cite{J1985}*{p. 78}. 
\begin{prop}\label{TopLevelParaproductCZO}
    If $b \in \text{BMO}(\R^{n_1} \times \R^{n_2})$, then $\Pi_b$ is a bi-parameter $1$-CZO with 
    $$
        \|\Pi_b\|_{1\text{CZO}(\mathbb{R}^{n_1}\times\mathbb{R}^{n_2})}\lesssim \|b\|_{BMO(\R^{n_1} \times \R^{n_2})}.
    $$
\end{prop}

\begin{thm}\label{TopLevelParaproductCompactness}
    If $b \in \text{CMO}(\R^{n_1} \times \R^{n_2})$, then $\Pi_b$, $\Pi_b^t$, $\Pi_{b,t}$, and $\Pi_{b,t}^t$ are compact on $L^2(\R^{n_1} \times \R^{n_2})$.\looseness=-1
\end{thm}
\begin{proof}
    Note that
    $$
        \overline{\DD(\R^{n_1}) \otimes \DD(\R^{n_2})}^{\text{BMO}(\R^{n_1} \times \R^{n_2})} \supseteq \overline{\DD(\R^{n_1}) \otimes \DD(\R^{n_2})}^{L^\infty(\R^{n_1} \times \R^{n_2})} = C_0(\R^{n_1} \times \R^{n_2}),
    $$
    which implies that $\overline{\DD(\R^{n_1}) \otimes \DD(\R^{n_2})}^{\text{BMO}(\R^{n_1} \times \R^{n_2})} = \text{CMO}(\R^{n_1} \times \R^{n_2})$. By a density argument and Proposition \ref{TopLevelParaproductCZO}, it suffices to show that $\Pi_{f \otimes g}$ is compact for $f \in \DD(\R^{n_1})$ and $g \in \DD(\R^{n_2})$. Since $\psi$ is the tensor product of two one-parameter wavelets and $\phi$ is the tensor product of one-parameter bump functions, we have
    $$
        \Pi_{f \otimes g} = \pi_f \otimes \pi_g,
    $$
    where $\pi_b$ is a classical one-parameter paraproduct. Since $\pi_f$, $\pi_f^t$, $\pi_g$, and $\pi_g^t$ are clearly compact on $L^2(\R^{n_i})$ for $f \in \DD(\R^{n_1})$ and $g \in \DD(\R^{n_2})$, it follows that any partial transpose of $\Pi_{f \otimes g} = \pi_f \otimes \pi_g$ is also compact on $L^2(\R^{n_1} \times \R^{n_2})$.
\end{proof}

\subsection{Partial paraproducts}

After subtracting the top-level paraproducts from a bi-parameter CZO, one is left with a top-level cancellative bi-parameter CZO, which is defined as follows.
\begin{defin}
    A bi-parameter CZO $T$ is top-level cancellative if 
$$
    \langle T(1 \otimes 1), \cdot \otimes \cdot\rangle \,  = \langle T(1 \otimes \cdot), \cdot \otimes 1\rangle = \langle T(\cdot \otimes 1), 1 \otimes \cdot\rangle
    = \langle T(\cdot \otimes \cdot), 1 \otimes 1\rangle = 0.
$$
\end{defin}

We next define the partial paraproducts. Below, $\pi_b$ is the one-parameter paraproduct with symbol $b$.  
    \begin{defin}
    The partial paraproducts $\varpi_{1,T}, \varpi_{2,T}\colon \mathcal{D}(\R^{n_1}) \otimes \mathcal{D}(\R^{n_2}) \rightarrow (\mathcal{D}(\R^{n_1}) \otimes \mathcal{D}(\R^{n_2}))'$ associated with a top-level cancellative bi-parameter $\delta$-CZO $T$ are given 
    by 
    $$
        \langle \varpi_{1,T} (f_1 \otimes f_2), g_1 \otimes g_2 \rangle = \langle \pi_{\langle T(1 \otimes f_2), \cdot \otimes g_2 \rangle} f_1 , g_1\rangle
    $$
    and
    $$
        \langle \varpi_{2,T} (f_1 \otimes f_2), g_1 \otimes g_2 \rangle = \langle \pi_{\langle T(f_1 \otimes 1), g_1 \otimes \cdot \rangle} f_2 , g_2\rangle.
    $$
    \end{defin}
    \noindent For the remainder of this section, we focus on $\varpi := \varpi_{1,T}$, the other being entirely symmetric. The relevant properties of the operator $\varpi$ are $\varpi_{(z,w)} = \pi_{\langle T(1 \otimes \psi_z), \cdot \otimes \psi_w \rangle}$, 
    $$
        \langle \varpi(1 \otimes \cdot), \cdot \otimes \cdot \rangle = \langle T(1 \otimes \cdot), \cdot \otimes \cdot \rangle, 
    $$
    and
    $$
        \langle \varpi(\cdot \otimes 1), \cdot \otimes \cdot \rangle = \langle \varpi(\cdot \otimes \cdot), 1 \otimes \cdot \rangle = \langle \varpi(\cdot \otimes \cdot), \cdot \otimes 1 \rangle = 0.
    $$
    
    We will need to know that the partial paraproducts associated with a top-level cancellative bi-parameter CZO are also bi-parameter CZOs to show that the fully cancellative operator obtained after subtracting off partial paraproducts is also a bi-parameter CZO. The following result was proved in \cite{J1985}*{Section 7}. 
    \begin{prop}
        The partial paraproduct $\varpi$ is a bi-parameter $\delta$-CZO. 
    \end{prop}

The following estimate provides the partial localization of the partial paraproducts.
\begin{lemma}\label{VarpiCoefficientsEstimate}
        If $z, w \in \mathbb{R}^{n_2+1}_+$, then 
        \[
            \| \varpi_{(z,w)}\|_{L^2(\R^{n_1}) \rightarrow L^2(\R^{n_1})} \lesssim L(z^{-1}w),
        \]
        where 
        \begin{align*}
            L((x,t)) = \begin{cases}
            \frac{1}{t^{n_2/2 + \delta}} & t \geq 1, |x| \leq t\\
            \frac{t^{n_2/2}}{|x|^{n_2+\delta}} & t \geq 1, |x| > t\\
            t^{n_2/2 + \delta} & t < 1, |x| \leq 1\\
            \frac{t^{n_2/2 + \delta}}{|x|^{n_2 + \delta}} & t < 1, |x|  > 1
        \end{cases}.
        \end{align*}
    \end{lemma}
    \begin{proof}
        We first claim that
        \begin{align*}
            \|\langle T(1 \otimes \psi_z), \cdot \otimes \psi_w \rangle \|_{\text{BMO}(\R^{n_1})} \lesssim L(z^{-1} w)
        \end{align*}
        for all $z,w \in \R_+^{n_2+1}$. Indeed, by $H^1$-BMO duality, we have 
        $$
            \|\langle T(1 \otimes \psi_z), \cdot \otimes \psi_w \rangle \|_{\text{BMO}(\R^{n_1})} \approx \sup_{\|a\|_{H^1(\R^{n_1})} \leq 1} |\langle T(1 \otimes \psi_z), a \otimes \psi_w \rangle|.
        $$
        Since $T$ is top-level cancellative, we have that $\langle T(1 \otimes \cdot ), a \otimes \cdot \rangle$ is a cancellative one-parameter $\delta$-CZO with $\delta\text{CZO}(\mathbb{R}^{n_1})$ norm controlled by $\|a\|_{H^1(\R^{n_1})}$ for each $a\in H^1(\R^{n_1}$), and so the claim follows from \cite{MS2023}*{Lemma 4.1}. We use the relation $\varpi_{(z,w)} = \pi_{\langle T(1 \otimes \psi_z), \cdot \otimes \psi_w \rangle}$ and properties of the one-parameter paraproduct to conclude that
        \[
            \|\varpi_{(z,w)}\|_{L^2(\R^{n_1}) \rightarrow L^2(\R^{n_1})} \lesssim \|\langle T(1 \otimes \psi_z), \cdot \otimes \psi_w \rangle\|_{\text{BMO}(\R^{n_1})} \lesssim L(z^{-1}w).
        \]
    \end{proof}

    The next result asserts that the partial paraproducts are partially localized. In particular, we apply the abstract theory of Section \ref{AbstractSection} with $H_1 = L^2(\R^{n_1})$, $H_2 = L^2(\R^{n_2})$, $G = \R_+^{n_2+1}$, $\psi_z = \psi_z$, and $E_N = D((0,1),N)$. Note that $L^2(\R^{n_1}\times\R^{n_2}) \cong L^2(\R^{n_1})\Ptens L^2(\R^{n_2})$.

    \begin{cor}\label{VarpiPartiallyWeaklyLocalized}
        If $T$ is a top-level cancellative bi-parameter CZO, then the associated partial paraproduct $\varpi$ is partially localized on $L^2(\R^{n_1}\times\R^{n_2})$ with respect to the frame $\{\psi_z\}_{z \in \R^{n_2+1}_+}$.
    \end{cor}
    \begin{proof}
        The result follows from Lemma \ref{VarpiCoefficientsEstimate} and Lemma \ref{MatrixBound>Localized}, since  
        $$
            \int_{\R^{n_2+1}_+} L(z) \omega(z) \, dz< \infty
        $$
        where $\omega(x,t) = t^{n_2/2}$.
    \end{proof}

    \begin{thm}\label{partialparaproductcompactness}
        If $T$ is a top-level cancellative bi-parameter CZO satisfying the mixed weak compactness/CMO condition, then $\varpi$ is compact on $L^2(\mathbb{R}^{n_1}\times\mathbb{R}^{n_2})$.  
    \end{thm}
    \begin{proof}
        Since $\varpi$ is partially localized by Corollary \ref{VarpiPartiallyWeaklyLocalized}, it suffices to show that $\varpi$ is weakly compact and apply Theorem \ref{Abstractbi-parameterCompactnessTheorem}. Recall that $\varpi_{(z,w)} = \pi_{\langle T(1 \otimes \psi_z), \cdot \otimes \psi_w \rangle}$, and that for each $w$ and $z$ this symbol is assumed to be in $\text{CMO}(\R^{n_1})$, so each $\varpi_{(z,w)}$ is compact. Also, 
        \[
            \|\varpi_{(z,w)}\|_{L^2(\R^{n_1}) \rightarrow L^2(\R^{n_1})} = \|\pi_{\langle T(1 \otimes \psi_z), \cdot \otimes \psi_w \rangle}\|_{L^2(\R^{n_1}) \rightarrow L^2(\R^{n_1})} \lesssim \|\langle T(1 \otimes \psi_z),\cdot \otimes \psi_w \rangle\|_{\text{BMO}(\R^{n_1})},
        \]
        so the vanishing assumption on $\langle T(1 \otimes \psi_z),\cdot \otimes \psi_w \rangle$ gives the rest of weak compactness of $\varpi$.\looseness=-1
    \end{proof}


\section{Compactness of fully cancellative bi-parameter SIOs}\label{FullyCancellativeSection}
    
    In this section, we apply the abstract theory of Section \ref{AbstractSection} in the case $H_1 = \C$, $H_2 = L^2(\R^{n_1} \times \R^{n_2})$, $G = \R_+^{n_1+1} \times \R_+^{n_2+1}$, $\psi_z = \psi_{z_1} \otimes \psi_{z_2}$, and $E_N = D((0,1),N) \times D((0,1),N)$. Note that $\psi_z = \mathfrak{U}(z) \psi :=  U(z_1) \otimes U(z_2) \psi$, where $U(x_i,t_i)$ is the unitary operator on $L^2(\R^{n_i})$ given by $U(x_i,t_i)f = t_i^{-n_i/2} f \big( \frac{\cdot - x_i}{t_i} \big)$. Also note that $\mathfrak{U}(z)^* = \mathfrak{U}(z)^t = \mathfrak{U}(z^{-1})$.

    \begin{defin}
    A bi-parameter CZO $T$ is fully cancellative if it is top-level cancellative and 
    \[
        \langle T(1 \otimes \cdot ) , \cdot \otimes \cdot \rangle =\langle T(\cdot \otimes 1 ) , \cdot \otimes \cdot \rangle = \langle T(\cdot \otimes \cdot ) , 1 \otimes \cdot \rangle = \langle T(\cdot \otimes \cdot ) , \cdot \otimes 1 \rangle = 0.
    \]
    \end{defin}
    \noindent We observe that this property is invariant under conjugation by $\mathfrak{U}(z)$ and that the kernels of fully cancellative operators are cancellative in the one-parameter sense.

    \begin{prop}\label{CancellationInvariance}
        If $T$ is a fully cancellative bi-parameter $\delta$-CZO and $w \in G$, then \\$\mathfrak{U}(w)^t T \mathfrak{U}(w)$ and all partial transposes of $T$ are also fully cancellative bi-parameter $\delta$-CZOs with equal $\delta\text{CZO}(\mathbb{R}^{n_1}\times\mathbb{R}^{n_2})$ norms. Further, both kernels of $T$ are pointwise cancellative one-parameter $\delta$-CZOs.\looseness=-1
    \end{prop}
    \begin{proof}
        The first statement is straightforward. For the second, we let $a \in H^1(\R^{n_1})$ and $f,g \in L^2(\R^{n_2})$ have disjoint compact supports. Then
        \[
            0 = \langle T(1 \otimes f), a \otimes g \rangle = \iint_{\R^{2n_2}} f(y) g(x) \langle K_2(x,y) 1, a \rangle \, dxdy.
        \]
        Since $f$ and $g$ were arbitrary, we have that $\langle K_2(x,y) 1 , a \rangle = 0$ almost everywhere and, by the continuity of $K_2$, everywhere. That is, $K_2(x,y)1 = 0$ for every $x,y \in \R^{n_2}$ with $x\neq y$. The remaining cancellation properties of the kernels follow from symmetric arguments.
    \end{proof}

    Before establishing the partial localization of the fully cancellative operators, let us recall the analogous one-parameter estimate from, e.g., \cite{MS2023}*{Lemma 4.1}.

    \begin{prop}\label{OneParameterAlmostDiag}
        If $T$ is a $\delta$-CZO satisfying $T1=T^t1=0$ and $(x,t) \in \mathbb{R}^{n+1}_+$, then 
$$
    |\langle T\psi,\psi_{(x,t)}\rangle|\lesssim
    \|T\|_{\delta \text{CZO}(\R^n)}\begin{cases}
    \frac{1}{t^{n/2+\delta}}&t\ge 1, \,\, |x|\leq t\\
    \frac{t^{n/2}}{|x|^{n+\delta}} & t\ge 1,\,\,|x|> t\\
    t^{n/2+\delta} & t< 1, \,\, |x|\leq 1\\
    \frac{t^{n/2+\delta}}{|x|^{n+\delta}} & t< 1, \,\, |x|> 1
    \end{cases}.
$$
    \end{prop}

    \begin{lemma}\label{FullyCancellativeCoefficientEstimate}
    If $T$ is a fully cancellative bi-parameter $\delta$-CZO, then
    \[
        \left| \langle T \psi, \psi_z \rangle \right| \lesssim \|T\|_{\delta \text{CZO}(\R^{n_1} \times \R^{n_2})} L(z),
    \]
    where $L(z) = L_1(z_1) L_2(z_2)$ for $z = (z_1,z_2)\in \mathbb{R}^{n_1+1}_+\times\mathbb{R}^{n_2+1}_+$ and
    \[
        L_i(x,t) = \begin{cases}
            \frac{1}{t^{n_i/2 + \delta}} & t \geq 1, |x| \leq t\\
            \frac{t^{n_i/2}}{|x|^{n_i+\delta}} & t \geq 1, |x| > t\\
            t^{n_i / 2 + \delta} & t < 1, |x| \leq 1\\
            \frac{t^{n_i / 2 + \delta}}{|x|^{n_i + \delta}} & t < 1, |x|  > 1
        \end{cases}.
    \]
    \end{lemma}

    \begin{proof}
    Assume $\|T\|_{\delta \text{CZO}(\R^{n_1} \times \R^{n_2)}} = 1$. We first claim that 
    $$
        \| \langle T(\cdot \otimes \psi), \cdot \otimes \psi_{z_2} \rangle \|_{\delta \text{CZO}(\R^{n_1})} \lesssim L_2(z_2).
    $$
    To estimate the $L^2(\mathbb{R}^{n_1})$ norm, fix $f,g \in L^2(\R^{n_1})$ with $\|f\|_{L^2(\R^{n_1})},\|g\|_{L^2(\R^{n_1})} \leq 1$. Then $\langle T(f \otimes \cdot ), g \otimes \cdot \rangle$ is a cancellative $\delta$-CZO, so by Proposition \ref{OneParameterAlmostDiag}, we have 
$$
    \sup_{\|f\|_{L^2(\R^{n_1})}, \|g\|_{L^2(\R^{n_1})} \leq 1} |\langle T(f \otimes \psi), g \otimes \psi_{z_2} \rangle| \lesssim L_2(z_2).
$$
The kernel of $\langle T(\cdot \otimes \psi), \cdot \otimes \psi_{z_2} \rangle$ is given by $\langle K_1(x,y) \psi, \psi_{z_2} \rangle$. Since $K_1(x,y)$ is cancellative for every $x \neq y$, we have $|\langle K_1(x,y) \psi, \psi_{z_2} \rangle| \lesssim |x-y|^{-n_1} L_2(z_2)$, again by Proposition \ref{OneParameterAlmostDiag}. A similar argument gives the smoothness estimate, and so we have the claim. Since this operator is cancellative, one last application of Proposition \ref{OneParameterAlmostDiag} gives
$$
    \left| \langle T \psi, \psi_z \rangle \right| \lesssim L_1(z_1) L_2(z_2),
$$
as desired.
\end{proof}

    \begin{cor}
        If $T$ is a fully cancellative bi-parameter CZO, then $T$ is partially localized with $H_1 = \C$, $H_2 = L^2(\R^{n_1} \times \R^{n_2})$, and frame $\{ \psi_z \}_{ z\in \R_+^{n_1+1} \times \R_+^{n_2+1}}$.
    \end{cor}
    \begin{proof}
        Note that $L(\cdot^{-1}) = L(\cdot)$, so by Lemma \ref{MatrixBound>Localized} with $\omega((x_1,t_1),(x_2,t_2)) = t_1^{n_1/2} t_2^{n_2/2}$, it is enough to show that
        \[
            |\langle T \mathfrak{U}(w) \psi , \mathfrak{U}(z) \psi \rangle| \lesssim L_1(w_1^{-1} z_1) L_2(w_2^{-1} z_2)
        \]
        or, equivalently, that 
        \[
            \sup_{w \in G} |\langle T_w \psi , \mathfrak{U}(z) \psi \rangle| \lesssim L_1( z_1) L_2( z_2),
        \]
        where $T_w := \mathfrak{U}(w)^t T \mathfrak{U}(w)$. This bound follows from Lemma \ref{FullyCancellativeCoefficientEstimate} and Proposition \ref{CancellationInvariance}.
    \end{proof}

    \begin{thm}\label{FullyCancellativeCompactness}
        If a fully cancellative bi-parameter CZO $T$ satisfies the product weak compactness property, then $T$ is compact on $L^2(\R^{n_1}\times\R^{n_2})$.
    \end{thm}
    \begin{proof}
        Since $H_1 = \C$, $T_{(z,w)}$ is a scalar, which is automatically compact on $\C$ for each $z,w \in \R_+^{n_1+1} \times \R_+^{n_2+1}$. Thus, the weak compactness property in this setting reduces to exactly the product weak compactness property, and the result follows from Theorem \ref{Abstractbi-parameterCompactnessTheorem}.
    \end{proof}

\section{Proof of Theorem \ref{bi-parameterCZOCompactnessTheorem}}\label{MainSection}

We now prove our main result, Theorem \ref{bi-parameterCZOCompactnessTheorem}.
\begin{proof}[Proof of Theorem \ref{bi-parameterCZOCompactnessTheorem}]
    Define
    $$
        T' := T - \Pi_{T1} - \Pi_{T_t1,t} - \Pi_{T_t^t1,t}^t - \Pi_{T^t1}^t.
    $$
    Note that $T'$ is a top-level cancellative bi-parameter CZO. By Theorem \ref{TopLevelParaproductCompactness} and the assumption $T1, T^t1, T_t1, T_t^t1 \in \text{CMO}(\mathbb{R}^{n_1}\times \mathbb{R}^{n_2})$, each paraproduct is a compact bi-parameter $1$CZO, so it suffices to prove the compactness of $T'$. Note that by Proposition \ref{WC/CMOnecessary} and Proposition \ref{productWCnecessary}, since each top-level paraproduct is compact, $T'$ also satisfies the same hypotheses as $T$. 
    
    Now define
    $$
        T'' := T' - \varpi_{1,T'} - \varpi_{1,T^{'t}}^t - \varpi_{2,T'} - \varpi_{2,T^{'t}}^t.
    $$
    Since the partial paraproducts are compact by Theorem \ref{partialparaproductcompactness}, it is enough to prove the compactness of $T''$. Further, note that the compactness of each $\varpi$ implies their product weak compactness by Proposition \ref{productWCnecessary}, and hence the product weak compactness of $T''$. Finally, since $T''$ is fully cancellative and satisfies the product weak compactness property, it is compact by Theorem \ref{FullyCancellativeCompactness}. 

    The necessity of the claimed conditions follows from Proposition \ref{productWCnecessary}, Proposition \ref{WC/CMOnecessary}, and Corollary \ref{nec3}.
\end{proof}

\begin{proof}[Proof of Corollary \ref{characterization}]
This immediately follows from Theorem \ref{bi-parameterCZOCompactnessTheorem} since $T_t$ is also a bi-parameter CZO.  
\end{proof}

\section{Acknowledgements}\label{Acknowledgements}

We thank Anastasios Fragkos and A. Walton Green for inspiring discussions and feedback.

This material is based upon work supported by the National Science Foundation Graduate Research Fellowship Program under Grant No. DGE 2139839. Any opinions, findings, and conclusions or recommendations expressed in this material are those of the authors and do not necessarily reflect the views of the National Science Foundation.



\begin{bibdiv}
\begin{biblist}
\bib{AM1986}{article}{
title={$H^p$ continuity properties of Calder\'on-Zygmund-type operators},
author={J. Alvarez},
author={M. Milman},
journal={J. Math. Anal. Appl.},
volume={118},
date={1986},
number={1},
pages={63--79},
review={\MR{0849442}}
}



\bib{BLOT2025}{article}{
title={Compact $T(1)$ theorem \`a la Stein},
author={\'A. B\'enyi},
author={G. Li},
author={T. Oh},
author={R. H. Torres},
journal={J. Funct. Anal.},
volume={289},
date={2025},
number={7},
pages={Paper No. 111052, 30 pp.},
review={\MR{4910488}}
}


\bib{BOT2024}{article}{
title={Symbolic calculus for a class of pseudodifferential operators with applications to compactness},
author={\'A. B\'enyi},
author={T. Oh},
author={R. H. Torres},
journal={J. Geom. Anal. },
number={10},
pages={Paper No. 301},
date={2025},
review={\MR{4941980}}
}

\bib{COY2022}{article}{
title={Extrapolation for multilinear compact operators and applications},
author={M. Cao},
author={A. Olivo},
author={K. Yabuta},
journal={Trans. Amer. Math. Soc},
volume={375},
date={2022},
number={7},
pages={5011--5070},
review={\MR{4439498}}
}

\bib{CYY2024}{article}{
title={A compact extension of Jorun\'e's $T1$ theorem on product spaces},
author={M. Cao},
author={K. Yabuta},
author={D. Yang},
journal={Trans. Amer. Math. Soc.},
volume={377},
date={2024},
number={9},
pages={6251--6309},
review={\MR{4855312}}
}

\bib{CST2023}{article}{
title={Extrapolation of compactness for certain pseudodifferential operators},
author={M. J. Carro},
author={J. Soria},
author={R. H. Torres},
journal={Rev. Un. Mat. Argentina},
volume={66},
date={2023},
number={1},
pages={177--186},
review={\MR{4653690}}
}

\bib{CF1980}{article}{
title={A continuous version of duality of $H^1$ with BMO on the bidisc},
author={S-Y. A. Chang},
author={R. Fefferman},
journal={Ann. of Math. (2)},
volume={112},
date={1980},
number={1},
pages={179--201},
review={\MR{0584078}}
}

\bib{CW1977}{article}{
title={Extensions of Hardy spaces and their use in analysis},
author={R. Coifman},
author={G. Weiss},
journal={Bull. Amer. Math. Soc},
volume={83},
date={1977},
number={4},
pages={569--645},
review={\MR{0447954}}
}

\bib{DJ1984}{article}{
    title={A boundedness criterion for generalized Calder\'on-Zygmund operators},
    author={G. David},
    author={J.-L. Journ\'e},
    journal={Ann. of Math. (2)},
    volume={120},
    date={1984},
    number={2},
    pages={371--397},
    review={\MR{0763911}}
}

\bib{DWW2023}{article}{
title={Wavelet representation of singular integral operators},
author={F. Di Plinio},
author={B. D. Wick},
author={T. Williams},
journal={Math. Ann.},
volume={386},
date={2023},
number={3-4},
pages={1829--1889},
review={\MR{4612408}}
}

\bib{MR0828217}{article}{
   author={R. Fefferman},
   title={Calder\'on-Zygmund theory for product domains: $H^p$ spaces},
   journal={Proc. Nat. Acad. Sci. U.S.A.},
   volume={83},
   date={1986},
   number={4},
   pages={840--843},
   review={\MR{0828217}}
}

\bib{FS1982}{article}{
title={Singular integrals on product spaces},
author={R. Fefferman},
author={E. M. Stein},
journal={Adv. Math.},
date={1982},
volume={45},
number={2},
pages={117--143},
review={\MR{0664621}}
}

\bib{FGW2023}{article}{
title={Multilinear wavelet compact $T(1)$ theorem},
author={A. Fragkos},
author={A. W. Green},
author={B. D. Wick},
journal={arXiv:2312.09185},
date={2023}
}

\bib{G2016}{article}{
title={Comparison of $T1$ conditions for multi-parameter operators},
author={A. Grau de la Herr\'an},
journal={Proc. Amer. Math. Soc.},
volume={144},
date={2016},
number={6},
pages={2437--2443},
review={\MR{3477059}}
}

\bib{HAN20102834}{article}{
title={Calder\'on-Zygmund operators on product Hardy spaces},
author={Y. Han},
author={M.-Y. Lee},
author={C.-C. Lin},
author={Y.-C. Lin},
journal={J. Funct. Anal.},
volume={258},
date={2010},
number={8},
pages={2834--2861},
review={\MR{2593346}}
}

\bib{HL2023}{article}{
title={Extrapolation of compactness on weighted spaces},
author={T. P. Hyt\"onen},
author={S. Lappas},
date={2023},
volume={39},
journal={Rev. Mat. Iberoam.},
number={1},
pages={91--122},
review={\MR{4571600}}
}

\bib{HNVW23}{book}{
title={Analysis in Banach spaces. Vol. III. Harmonic analysis and spectral theory},
author={T. Hytönen},
author={J. van Neerven},
author={M. Veraar},
author={L. Weis},
series={Ergeb. Math. Grenzgeb. (3), 76},
publisher={Springer, Cham},
series={[Results in Mathematics and Related Areas. $3^{\text{rd}}$ Series. A Series of Modern Surveys in Mathematics]},
date={2023},
pages={xxi+826 pp.},
review={\MR{4696978}}
}

\bib{J1985}{article}{
title={Calder\'on-Zygmund operators on product spaces},
author={J.-L. Journ\'e},
journal={Rev. Mat. Iberoam.},
volume={1},
date={1985},
number={3},
pages={55--91},
review={\MR{0836284}}
}

\bib{M2012}{article}{
title={Representation of bi-parameter singular integrals by dyadic operators},
author={H. Martikainen},
journal={Adv. Math.},
volume={229},
date={2012},
number={3},
pages={1734--1761},
review={\MR{2871155}}
}

\bib{M1992}{book}{
title = {Wavelets and operators},
author = {Y. Meyer},
series = {Cambridge Stud. Adv. Math., 37},
isbn = {0521420008},
year = {1992},
publisher = {Cambridge University Press},
address={Cambridge},
pages={xvi+224 pp.},
review={\MR{1228209}}
}


\bib{MS2023}{article}{
title={On the $T1$ theorem for compactness of Calder\'on-Zygmund operators},
author={M. Mitkovski},
author={C. B. Stockdale},
journal={arXiv:2309.15819},
date={2023}
}

\bib{MSWW2023}{article}{
title={Riesz-Kolmogorov type compactness criteria in function spaces with applications},
author={M. Mitkovski},
author={C. B. Stockdale},
author={N. A. Wagner},
author={B. D. Wick},
journal={Complex Anal. Oper. Theory},
volume={17},
date={2023},
number={3},
pages={Paper No. 40},
review={\MR{4569081}}
}

\bib{OV2017}{article}{
title={Endpoint estimates for compact Calder\'on-Zygmund operators},
author={J-F. Olsen},
author={P. Villarroya},
journal={Rev. Mat. Iberoam.},
volume={33},
date={2017},
pages={1285–-1308},
review={\MR{3729600}}
}

\bib{O2015}{article}{
title={A $T(b)$ theorem on product spaces},
author={Y. Ou},
journal={Trans. Amer. Math. Soc.},
volume={367},
date={2015},
number={9},
pages={6159--6197},
review={\MR{3356933}}
}

\bib{O2017}{article}{
title={Multi-parameter singular integral operators and representation theorem},
author={Y. Ou},
journal={Rev. Mat. Iberoam.},
volume={33},
date={2017},
number={1},
pages={325--350},
review={\MR{3615454}}
}


\bib{PPV2017}{article}{
title={Endpoint compactness of singular integrals and perturbations of the Cauchy integral},
author={K-M. Perfekt},
author={S. Pott},
author={P. Villarroya},
journal={Kyoto J. Math.},
volume={57},
date={2017},
number={2},
pages={365--393},
review={\MR{3648054}}
}

\bib{PV2011}{article}{
title={A $T(1)$ theorem on product space},
author={S. Pott},
author={P. Villarroya},
journal={arXiv:1105.2516},
date={2011}
}

\bib{SVW2022}{article}{
title={Sparse domination results for compactness on weighted spaces},
author={C. B. Stockdale},
author={P. Villarroya},
author={B. D. Wick},
journal={Collect. Math.},
volume={73},
date={2022},
number={3},
pages={535--563},
review={\MR{4467913}}
}

\bib{V2015}{article}{
    title={A characterization for compactness of singular integrals},
    author={P. Villarroya},
    journal={J. Math. Pures Appl. (9)},
    volume={104},
    date={2015},
    number={3},
    pages={485--532},
    review={\MR{3383175}}
}
\end{biblist}
\end{bibdiv}
\end{document}